\documentclass[numbook, envcountsect, envcountsame, envcountreset, runningheads, smallextended]{svjour3}
\smartqed
\makeatletter
\let\svjour@endproof\endproof
\def\endproof{\qed\svjour@endproof}
\makeatother

\usepackage{graphicx}
\usepackage{amsmath,amssymb}
\usepackage{booktabs}
\usepackage{float}
\usepackage{times}

\newcommand{\bbr}{\mathbb{R}}

\newcommand{\bbz}{\mathbb{Z}}
\newcommand{\bbc}{\mathbb{C}}
\newcommand{\bbp}{\mathbb{P}}
\newcommand{\bbe}{\mathbb{E}}
\newcommand{\wt}{\widetilde}
\newcommand{\wh}{\widehat}
\newcommand{\eps}{\varepsilon}
\newcommand{\dd}{\mathrm{d}}
\newcommand{\Ind}{\mathbf{1}}
\newcommand{\cA}{\mathcal{A}}
\newcommand{\cF}{\mathcal{F}}
\newcommand{\cH}{\mathcal{H}}
\newcommand{\cI}{\mathcal{I}}
\newcommand{\cK}{\mathcal{K}}
\newcommand{\cL}{\mathcal{L}}
\newcommand{\cP}{\mathcal{P}}
\DeclareMathOperator{\ord}{ord}
\DeclareMathOperator*{\argmin}{arg\,min}

\makeatletter
\def\jelclassname{{\bfseries JEL Classification}\enspace}
\def\jelclass#1{\par\addvspace\medskipamount{\rightskip=0pt plus1cm
\def\and{\ifhmode\unskip\nobreak\fi\ $\cdot$
}\noindent\jelclassname\ignorespaces#1\par}}
\makeatother

\begin{document}

\title{Exact solution of the ruin problem with proportional investments
and the limits of applicability of the power-law asymptotics}

\titlerunning{Exact solution of the ruin problem with proportional investments}

\author{Platon Promyslov \and Maxim Romanov \and Goluba Yurieva}

\authorrunning{P. Promyslov, M. Romanov, G. Yurieva}

\institute{Platon Promyslov \at
    HSE University, Faculty of Computer Science, Moscow, Russia \\
    \email{platon.promyslov@gmail.com}
    \and
    Maxim Romanov \at
    Lomonosov Moscow State University, Faculty of Mechanics and Mathematics, Moscow, Russia \\
    \email{mcliz@mail.ru}
    \and
    Goluba Yurieva \at
    Federal Research Center ``Computer Science and Control'' of the Russian Academy of Sciences, Moscow, Russia \\
    \email{goluba.yurieva@gmail.com}
}

\date{Received: date / Accepted: date}

\maketitle

\begin{abstract}
We consider the Cram\'er--Lundberg model with exponential claims and a constant fraction of the reserve invested in a risky asset. The integro-differential equation for the survival probability reduces to a doubly confluent Heun equation, which yields an explicit expression for it at any initial capital. Positivity, monotonicity, and boundedness of the solution follow from a first integral, without probabilistic arguments. We identify the power-law tail constant with a connection coefficient and, for an integer tail exponent, give it in closed form. Corrections to the leading term give an explicit validity threshold, and the asymptotically optimal fraction exceeds the exact one by a factor of roughly two to four across the parameter range considered.
\jelclass{G22 \and C63}
\end{abstract}

\section{Introduction}
\label{sec:intro}

In classical collective risk theory the reserve is a process with constant drift perturbed by a Poisson stream of claims. A modern generalization allows the reserve to be invested in a financial market. The inclusion of a risky asset changes the dynamics qualitatively: even with light (for instance, exponential) claim tails, the ruin probability acquires a heavy power-law decay~\cite{Frolova2002,Kalashnikov2002,Paulsen1993}, and for sufficiently large volatility ruin occurs almost surely~\cite{Paulsen1998}.

There are two analytical approaches to studying the ruin probability with investments. The first rests on implicit renewal theory~\cite{Goldie1991} and on distributional equations of the form $X=AX+B$~\cite{BDM}. It yields the exact decay rate for wide classes of processes~\cite{Albrecher2012,Eberlein2022,Ellanskaya2021,KLP2025,KPro2023,Pergamenshchikov2006,Promyslov2026}. The reach of the probabilistic technique is broad: for claims with regularly varying tails and returns of general form, asymptotics of the ruin probability on a finite horizon has been obtained, uniform over the whole horizon including the infinite one~\cite{ChengKW2022}. Closely related results under arbitrary dependence between insurance and financial risks appear in~\cite{XuPengZou2025,YangFanYuen2023}. The numerical experiments reported there show that these estimates are meaningful also at moderate initial capitals: in~\cite[Table~1]{ChengKW2022} the ratio of the simulation result to the asymptotic estimate ranges from $0.76$ at $x=100$ to $1.01$ at $x=900$. All such results, however, describe the behavior as $x\to\infty$ and give no value of the survival probability at a given finite capital level. The second approach is the analysis of the integro-differential equation (IDE) which the survival probability satisfies by virtue of the generalized It\^o formula and martingale arguments. In this direction, smoothness, existence, and uniqueness of classical and viscosity solutions have been studied~\cite{AK2026,AKvisc,BelkinaKabanov,Grandits2004,KPuk,P2026}. This paper follows the second approach.

The difference in the nature of the tail also explains what happens to the constant. For claims with a regularly varying tail the exponent is inherited from the claim distribution, and by Breiman's lemma the constant is a moment of the discounting factor~\cite[Theorem~3.1]{ChengKW2022}. Here the claims are exponential and have no power-law tail at all: the exponent $\gamma-1$, where $\gamma$ is the structural parameter~\eqref{eq:structural} defined in Section~\ref{sec:model}, is generated by the investment risk and is a Cram\'er root (Remark~\ref{rem:cramer}), while the constant becomes a connection coefficient of an equation with an accessory parameter (Remark~\ref{rem:srank}), which is why it remains uncomputed.

The equation itself goes back to Frolova's conference abstracts~\cite{Frolova1994}, where, for exponentially distributed claims, it was reduced to a linear ordinary differential equation (ODE) and a conjecture on the power-law decay of the ruin probability was stated. The conjecture was later proved in~\cite{Frolova2002} for non-life insurance and in~\cite{KabPerg2016} for negative risk sums. Grandits~\cite{Grandits2004} derived the IDE under weak regularity assumptions and proved a Karamata-type theorem. For exponentially distributed claims the reduction of the IDE to a second-order ODE for the survival density is written out in~\cite[\S6]{KabPerg2016} and~\cite[\S7]{KPuk}.

Especially relevant for us is the series of works treating the same equation as a singular boundary value problem. Belkina, Konyukhova, and Kurochkin~\cite{BKK2012,BKK2014,BKK2016} reduced our model to a third-order ODE, established the irregularity of both singular points of rank one, and proved the correspondence between the boundary value problem and the original IDE problem. Sufficiency theorems were obtained by Belkina~\cite{BelkinaSuff}. Constantinescu and Thomann~\cite{Albrecher2012,CT2004} work with the Laplace transform of the IDE instead of the IDE itself, and there~\cite[Remark~3.2]{Albrecher2012} the computation of the constant is pointed out as an open question. Antipov~\cite{AntipovLaplace} extended this technique to jumps of both signs with densities having rational Laplace transforms: the Laplace transform of the derivative of the ruin probability satisfies there a second-order ODE with a regular singular point at zero and exponents $0$ and $2a/\sigma^{2}-1$, whence, for $1<2a/\sigma^{2}<2$, a power-law asymptotics follows by a Tauberian theorem. There too the constant remains uncomputed.

The model, the IDE~\eqref{eq:IDE}, the passage from the IDE to the second-order ODE~\eqref{eq:ODE_final}, and the power-law decay rate $\Psi(u)=O\bigl(u^{-(\gamma-1)}\bigr)$ of Proposition~\ref{prop:ruin_asymptotics} are known. All of this is already contained in~\cite{Frolova2002}, where the entire reserve is assumed to be invested in the risky asset and there is no risk-free asset: equation~(12) of that paper is our IDE~\eqref{eq:IDE}, and equation~(13), of second order in $G=\Phi'$, coincides with~\eqref{eq:ODE_final} up to notation. Indeed, the coefficients written out in~\cite[p.~234]{Frolova2002} are $p(u)=\mu+2\bigl(1+a/\sigma^{2}\bigr)u^{-1}+2c\sigma^{-2}u^{-2}$ and $q(u)=2a\mu\sigma^{-2}u^{-1}+2\sigma^{-2}\bigl(a-\alpha+c\mu\bigr)u^{-2}$, which for $\gamma=2a/\sigma^{2}$, $\beta=2c/\sigma^{2}$, $\nu=2\alpha/\sigma^{2}$ gives exactly~\eqref{eq:pq}. We obtain our model by the change of parameters~\eqref{eq:effective}. The asymptotic result is Theorem~1 of the same paper. The structure of the singular points of the associated ODE appears in~\cite{BKK2016}. The constant in the power-law asymptotics is available in a ``computable'' form: reducing the IDE to a pair of integral equations, the second of which is a Volterra equation of the second kind, Promyslov expresses it in terms of the solutions of these equations~\cite{Promyslov2601}, and for the non-life model this is done under minimal assumptions on the claim distribution~\cite{P2026} (see also~\cite{AK2026}).

The ruin problem under risky investments remains the subject of an ongoing series of works: besides the foundational paper~\cite{Frolova2002}, it is treated in~\cite{KabPerg2016} for negative risk sums, in~\cite{KPro2023} for the Sparre Andersen model with annuity payments, and in~\cite{AKvisc} for a finite interval in terms of viscosity solutions.

The contribution of the paper consists of the following five items.

\begin{enumerate}
\item[\rm(i)] We identify the ODE~\eqref{eq:ODE_final} with the doubly confluent Heun equation in the canonical form~\cite[(31.12.2)]{DLMF}, with the parameters given explicitly (Proposition~\ref{prop:DCHE}), and express the survival probability in closed form, for any initial capital $u\ge0$, in terms of the unique solution of this equation bounded at the irregular origin (Theorem~\ref{thm:exact_solution}). To our knowledge, this is the first expression of the survival probability in terms of a solution of a Heun equation in a setting with both a volatile asset ($\sigma_\kappa>0$) and a compound Poisson stream of claims. The closest explicit result~\cite{PaulsenGjessing1997} concerns other settings: the formulas for jump claims there cover deterministic returns (Examples~2.1--2.4), and the only case with volatile returns (\S3) is purely diffusive, without jumps. When both conditions hold at once, the preceding literature provides either asymptotic formulas or numerical solutions of singular boundary value problems.

\item[\rm(ii)] We prove analytically all the properties needed to identify the constructed function with the probability: the solution $\cH$ is strictly positive (Lemma~\ref{lem:positivity}), whence monotonicity, strict positivity of the normalization constant, and the bounds $0<\wt\Phi\le1$ (Corollary~\ref{cor:admissible}), and the connection coefficient of the power-law mode is strictly positive (Lemma~\ref{lem:K1pos}). We establish these properties analytically rather than derive them from the probabilistic meaning of the constructed function, and that is what frees the verification theorem of Section~\ref{sec:verification} from logical circularity. The same first integral yields the converse implication ODE${}\Rightarrow{}$IDE (Proposition~\ref{prop:equivalence}): the solution space of~\eqref{eq:ODE_final} is two-dimensional, whereas the IDE singles out a one-dimensional linear family, and we compute the residual explicitly.

\item[\rm(iii)] We identify the constant $\cK$ in $\Psi(u)\sim\cK u^{-(\gamma-1)}$ with the central connection coefficient of the explicit second-order equation satisfied by the Laplace transform of the survival density (Proposition~\ref{prop:laplace} and Theorem~\ref{thm:connection}). The tail exponent is one of the exponents of this equation at zero (an observation of Antipov~\cite{AntipovLaplace}, whose reduction Proposition~\ref{prop:laplace} reproduces in the one-sided case), and the claim intensity to be its only accessory parameter (Remark~\ref{rem:srank}). We therefore expect no closed Gauss-type expression. For integer $\gamma$ the connection problem can nevertheless be solved in closed form: Theorem~\ref{thm:closure} expresses $\cK$, and with it the connection coefficient of the Heun equation, in terms of the single quadrature $\int_0^{\infty}\cH$, which we compute anyway in the normalization. This complements the computable representations of~\cite{Promyslov2601,P2026,AK2026}, where the constant appears as the limit of a solution of a Volterra equation of the second kind, and gives a partial answer, for this model and integer $\gamma$, to the question posed in~\cite[Remark~3.2]{Albrecher2012}.

\item[\rm(iv)] We compute the first correction to the power-law mode (Proposition~\ref{prop:validity}): the classical asymptotics is quantitatively applicable only for $u\gg|\mu\beta-\nu|/\mu=2|c\mu-\lambda|/(\mu\sigma_\kappa^{2})$. Divided by the premium rate, the threshold equals $u^{\ast}/c=2|\theta|/\bigl((1+\theta)\sigma_\kappa^{2}\bigr)$ and has the dimension of time: it is the period over which the insurer collects a premium equal to the threshold. For this reason it is independent of the monetary unit, and among the model parameters only the safety loading and the portfolio volatility enter it: neither the average claim size nor the line of business plays any role in it (Remark~\ref{rem:rubles}). For a conservative portfolio this period amounts to several decades, and for moderate volatility to some eight years, whereas the insurer's capital corresponds to one or two years of premiums. The threshold therefore exceeds a realistic capital severalfold, and for a conservative portfolio by an order of magnitude and more, so that the power-law tail, while remaining the true asymptotics, is not observable at such a capital. This delineates the practical domain of applicability of the asymptotics in this model. Theorem~\ref{thm:fullseries} establishes that both formal series at infinity, the power-law one and the recessive one, are complete Poincar\'e-type asymptotic expansions with explicit recurrence relations. This turns the applicability threshold from a heuristic into a quantity that can be verified at any truncation order $N$, and not only from the first correction.

\item[\rm(v)] From this follows the conclusion for the sake of which the exact solution is needed. For constant-fraction strategies, the natural comparison criterion is the power-law asymptotics~\eqref{eq:tail}. Its constant has, in general, no explicit expression. We show that this asymptotics does not suffice for choosing the fraction itself, even if the constant were known exactly. Choosing a strategy from the leading term of the asymptotics is not an invented device: in this literature it is known as an asymptotically optimal strategy. The form of the criterion depends on what in the leading term is sensitive to the strategy. For light-tailed claims this is the exponent of the exponential decay, and it is maximized by a constant investment sum: every Markov strategy asymptotically separated from it gives an exponentially worse decay~\cite{GGS2003}. For heavy tails the exponent is fixed by the claim distribution and does not depend on the strategy, so strategies are compared by the constant at the leading term. The asymptotically optimal strategy on a finite horizon is then a constant fraction of the reserve~\cite{HultLindskog2011}. In our problem both the exponent and the constant depend on $\kappa$, so the whole leading term must be minimized at a given $u$. The class of constant fractions is itself not optimal: in the problem with feedback the optimal investment sum depends on the reserve~\cite{HippPlum2000,HippPlum2003} and only tends to a constant as $u\to\infty$~\cite{HippSchmidli2004}, while for subexponential claims Schmidli~\cite{Schmidli2005} obtained the asymptotics of the ruin probability under the exactly optimal strategy. We, however, compare not classes of strategies but two ways of choosing the fraction $\kappa$ within the same class. Minimizing the leading term $\cK(\kappa)u^{-(\gamma(\kappa)-1)}$ over $\kappa$ gives, at $u=35$ average claims, where the ruin probability is of order $5\cdot10^{-3}$, the recommendation to keep $54\%$ of the reserve in the risky asset, whereas the exact solution gives $24\%$. At $u=60$ it is $37\%$ versus $14\%$ (Section~\ref{sec:numerics}). The discrepancy persists: under variation of the loading, the volatility, the risk premium, and the claim intensity, the ratio $\kappa_{\mathrm{as}}/\kappa^{\ast}$ stays between $1.6$ and $4.0$, while the fractions themselves vary by a factor of five (Remark~\ref{rem:robust}). The reason is structural: the optimal fraction is small, while the applicability threshold~\eqref{eq:validity} grows as $\kappa^{-2}$, so that the minimum lies well outside the region where the leading term describes the solution. What the asymptotics gets wrong is thus primarily the strategy rather than the level of risk.
\end{enumerate}

The optimal fraction $\kappa^{\ast}(u)$ turns out to depend on the capital: the minimum lies in the interior of $(0,1]$ and shifts to the left as $u$ grows, and at $u=0$ we establish the existence of an interior minimizer analytically in Proposition~\ref{prop:kappastar0}. We have verified the numerical results in the region $\Psi\gtrsim10^{-2}$, where direct simulation is informative, by the Monte Carlo method.

The paper is organized as follows. Section~\ref{sec:model} formalizes the model. In Section~\ref{sec:reduction} we reduce the IDE to an ODE. Section~\ref{sec:classification} analyzes the singular points and identifies the equation. In Section~\ref{sec:properties} we establish the properties of the solution of this equation bounded at zero. In Section~\ref{sec:verification} we prove the verification theorem, which relies on no property of the survival probability. Section~\ref{sec:solution} contains the exact formula and the asymptotics. In Section~\ref{sec:laplace} we identify the constant in the power-law asymptotics with a connection coefficient and clarify the limits of applicability of the asymptotics itself. Section~\ref{sec:numerics} contains a numerical illustration.

\section{Model and problem statement}
\label{sec:model}

Let a stochastic basis $(\Omega,\cF,\mathbf F,\bbp)$ with a filtration $\mathbf F=(\cF_t)_{t\ge0}$ satisfying the usual conditions be given. On it, the following independent objects are defined: a standard Brownian motion $W=(W_t)_{t\ge0}$, a homogeneous Poisson process $N=(N_t)_{t\ge0}$ of intensity $\lambda>0$, and a sequence $\xi=(\xi_k)_{k\ge1}$ of independent identically distributed claim sizes, independent of $W$ and $N$, with exponential distribution function $F(x)=1-e^{-\mu x}$, $x\ge0$, $\mu>0$. We set $\overline F(x)=e^{-\mu x}$ and $\bbe\xi_1=1/\mu$. The aggregate claims up to time $t\ge0$ form a compound Poisson process
\begin{equation*}
    Z_t=\sum_{k=1}^{N_t}\xi_k .
\end{equation*}

Consider an insurance company with initial capital $u\ge0$. Premiums arrive deterministically at a constant rate $c>0$. A fixed fraction $\kappa\in(0,1]$ of the reserve is invested in the risky asset, and the remaining part $1-\kappa$ is kept in a risk-free bank account. The prices of the risk-free and risky assets obey the equations
\begin{equation*}
    \dd B_t=rB_t\,\dd t,\qquad \dd S_t=S_t\bigl(a\,\dd t+\sigma\,\dd W_t\bigr),
\end{equation*}
where $r\ge0$ is the risk-free interest rate, $a>r$ is the expected return of the risky asset, and $\sigma>0$ is its volatility. The reserve $X=(X_t)_{t\ge0}$ with $X_0=u$ satisfies the stochastic differential equation with jumps
\begin{equation*}
    \dd X_t=(c+a_\kappa X_t)\,\dd t+\sigma_\kappa X_t\,\dd W_t-\dd Z_t,\qquad X_0=u,
\end{equation*}
where the effective parameters of the investment portfolio are
\begin{equation}
    a_\kappa:=r+\kappa(a-r),\qquad \sigma_\kappa:=\kappa\sigma .
    \label{eq:effective}
\end{equation}
The degenerate case $\kappa=0$ is excluded, so that $\sigma_\kappa>0$. We define the time of ruin as
\begin{equation*}
    \tau=\inf\{t>0:\ X_t<0\},\qquad \inf\varnothing=\infty.
\end{equation*}
The object of study is the survival probability on an infinite time horizon, defined as a function of the initial capital,
\begin{equation*}
    \Phi(u):=\bbp(\tau=\infty\mid X_0=u),\qquad u\ge0,
\end{equation*}
together with the ruin probability $\Psi(u):=1-\Phi(u)$.

Throughout the paper, for $f,g>0$ we write $f(u)\sim g(u)$ if $f(u)/g(u)\to1$ and $f(u)=O[g(u)]$ if $\limsup f(u)/g(u)<\infty$, in both cases as $u\to\infty$, unless stated otherwise. Also, $x\wedge y:=\min\{x,y\}$.

\begin{lemma}[Behavior at the origin]
\label{lem:origin}
Between two consecutive claims the reserve coincides with the solution $Y$ of the equation $\dd Y_t=(a_\kappa Y_t+c)\dd t+\sigma_\kappa Y_t\dd W_t$, $Y_0=y\ge0$. Then:
\begin{enumerate}
\item[\rm(i)] this solution is written out explicitly,
\begin{equation}
    Y_t=\mathcal E_t\Bigl(y+c\int_0^t\mathcal E_s^{-1}\,\dd s\Bigr),
    \qquad
    \mathcal E_t:=\exp\Bigl\{\bigl(a_\kappa-\tfrac12\sigma_\kappa^2\bigr)t
                              +\sigma_\kappa W_t\Bigr\};
    \label{eq:Yexpl}
\end{equation}
\item[\rm(ii)] $Y_t>0$ for all $t>0$ and every $y\ge0$, and for $y>0$ one has $Y_t\ge y\,\mathcal E_t>0$ for all $t\ge0$;
\item[\rm(iii)] $X_\tau<0$ almost surely on $\{\tau<\infty\}$, and in particular $\bbp(X_\tau=0)=0$.
\end{enumerate}
\end{lemma}

\begin{proof}
We verify formula~\eqref{eq:Yexpl} by the It\^o formula. Since $\mathcal E_t>0$ and $c>0$, the integral in~\eqref{eq:Yexpl} is nonnegative and strictly positive for $t>0$. This gives both inequalities in~(ii) immediately. Hence the continuous component of the reserve never reaches the origin, and the reserve can become negative only at a jump time. At such a time the claim is strictly positive and has a continuous distribution, so $X_\tau<0$ a.s.\ on $\{\tau<\infty\}$, and the equality $\bbp(X_\tau=0)=0$ follows from this inequality.
\end{proof}

\begin{remark}[Monotonicity of the survival probability]
\label{rem:monotone}
The function $\Phi$ is nondecreasing on $[0,\infty)$. Indeed, take $0\le u<u'$ and construct both reserve processes from the same $W$ and the same claim stream. The difference $D:=X^{u'}-X^{u}$ satisfies the equation $\dd D_t=D_t(a_\kappa\dd t+\sigma_\kappa\dd W_t)$ with $D_0=u'-u>0$, since the drift $c\,\dd t$ and the jumps cancel. Therefore $D_t=(u'-u)\mathcal E_t>0$ for all $t$, so that $X^{u'}_t>X^{u}_t$ pathwise, $\tau^{u}\le\tau^{u'}$, and $\Phi(u)\le\Phi(u')$. This remark explains the term ``survival density'' for $g=\Phi'$ in Section~\ref{sec:reduction}. In Section~\ref{sec:verification} it is not used.
\end{remark}

For the problem to be nondegenerate (that is, for ruin not to occur almost surely at every initial capital) the asymptotic drift of the logarithm of the geometric Brownian motion must be positive:
\begin{equation}
    a_\kappa-\tfrac12\sigma_\kappa^2>0 .
    \label{eq:nondegeneracy}
\end{equation}
This condition, well known from the works~\cite{Paulsen1993,Paulsen1998}, guarantees that in the absence of jumps $X_t\to+\infty$ $\bbp$-a.s., and under it $\lim_{u\to\infty}\Phi(u)=1$. The latter relation is quoted here for reference and is not used below: in Section~\ref{sec:verification} the equality $\wt\Phi(\infty)=1$, which follows from the normalization~\eqref{eq:Cdef}, takes its place.

It is convenient to introduce
\begin{equation}
    \gamma:=\frac{2a_\kappa}{\sigma_\kappa^{2}},\qquad
    \beta:=\frac{2c}{\sigma_\kappa^{2}},\qquad
    \nu:=\frac{2\lambda}{\sigma_\kappa^{2}} .
    \label{eq:structural}
\end{equation}
In this notation~\eqref{eq:nondegeneracy} takes the form $\gamma>1$, and $\beta,\nu>0$. The quantities $\gamma$ and $\nu$ are pure numbers: the drift $a_\kappa$, the intensity $\lambda$, and the variance $\sigma_\kappa^{2}$ carry units inverse to time, and in the ratios~\eqref{eq:structural} these units cancel, so that $\gamma$ and $\nu$ do not depend on the choice of the monetary unit. The quantity $\beta$, however, carries the same units as the capital: the premium $c$ is money per unit time, so that on passing from dollars to thousands of dollars $\beta$ is divided by a thousand. The product $\mu\beta$ is a pure number, since the average claim size $1/\mu$ carries the same units. It is $\mu\beta$, not $\beta$, that enters all the formulas below. There are thus exactly three essential dimensionless parameters,
\begin{equation}
    (\gamma,\ \mu\beta,\ \nu),
    \qquad\text{with}\qquad
    \frac{\nu}{\mu\beta}=\frac{\lambda\bbe\xi_1}{c}=\frac{1}{1+\theta},
    \label{eq:threeparams}
\end{equation}
where $\theta:=c\mu/\lambda-1$ is the relative safety loading. The numerical examples of Section~\ref{sec:numerics} use consistent units and a strictly positive safety loading.

\section{Reduction of the IDE to an ODE}
\label{sec:reduction}

By Theorem~1 of~\cite{P2026}, under condition~\eqref{eq:nondegeneracy} the survival probability $\Phi$ belongs to $C^2((0,\infty))$, and $\Phi'$ is absolutely continuous on $[0,\infty)$. For exponentially distributed claims this also follows from Theorem~5.1 of~\cite{KPuk}. Hence $\Phi$ is a classical solution on $(0,\infty)$ of the linear IDE
\begin{equation}
    \tfrac12\sigma_\kappa^2u^2\Phi''(u)+(a_\kappa u+c)\Phi'(u)
    =\lambda\Phi(u)\overline F(u)
     -\lambda\int_0^u\bigl(\Phi(u-y)-\Phi(u)\bigr)\dd F(y),
    \label{eq:IDE}
\end{equation}
where all integrals of the form $\int_a^b$ are understood over $(a,b]$. Extending $\Phi$ by zero to the negative half-axis and setting
\begin{equation}
    \cA\Phi(u):=\tfrac12\sigma_\kappa^2u^2\Phi''(u)+(a_\kappa u+c)\Phi'(u)
                -\lambda\Phi(u)+\lambda\int_0^u\Phi(u-y)\,\dd F(y),
    \label{eq:generator}
\end{equation}
equation~\eqref{eq:IDE} is written as $\cA\Phi=0$ on $(0,\infty)$. In~\eqref{eq:generator} itself the extension changes nothing: for $u>0$ the value $\Phi(u-y)$ is taken only for $y\in(0,u]$, and negative arguments do not occur there. Section~\ref{sec:verification} needs it. The It\^o formula generates an integral over all claim sizes $y\in(0,\infty)$, including those with $y>u$, which lead to ruin, and at these points the value $\Phi(u-y)$ must be defined. The equality $\Phi\equiv0$ on the negative half-axis expresses analytically the fact that after ruin the survival probability is zero. Thanks to it, $\int_0^{u}\Phi(u-y)\,\dd F(y)=\int_0^{\infty}\Phi(u-y)\,\dd F(y)$, the operator $\cA$ coincides with the generator of the reserve stopped at ruin, and the drift in the It\^o formula vanishes by virtue of $\cA\Phi=0$. The third use is the terminal value: on $\{\tau<\infty\}$ the reserve turns out to be strictly negative, and the extension gives zero there. One cannot write $\cA\Phi=0$ on the whole line: for $u<0$ the equality reduces to the identity $0=0$, and at the point $u=0$ the extended function is discontinuous, since $\Phi(0+)=1-\Psi(0)>0$, and the second derivative does not exist there.

\begin{lemma}[Bootstrap of regularity]
\label{lem:bootstrap}
For exponentially distributed claims $\Phi\in C^\infty((0,\infty))$.
\end{lemma}

\begin{proof}
By Theorem~1 of~\cite{P2026} we have $\Phi\in C^2((0,\infty))$, and $\Phi$ is bounded. Hence $\cI(u):=\lambda\mu e^{-\mu u}\int_0^u\Phi(z)e^{\mu z}\dd z$ belongs to $C^1$, and $\cI'=-\mu\cI+\lambda\mu\Phi$, so that $\cI'$ is again a continuous function of $\Phi$ and~$\cI$. Assume by induction that $\Phi\in C^{k}$ and $\cI\in C^{k}$ for some $k\ge1$. Then the right-hand side of the identity for $\cI'$ belongs to $C^{k}$, whence $\cI\in C^{k+1}$. Solving~\eqref{eq:IDE} for the second derivative, $\Phi''=\frac{2}{\sigma_\kappa^2u^2}\bigl[\lambda\Phi-(a_\kappa u+c)\Phi'-\cI\bigr]$, we obtain $\Phi''\in C^{k-1}$ on $(0,\infty)$, that is, $\Phi\in C^{k+1}$. This completes the induction step.
\end{proof}

Lemma~\ref{lem:bootstrap} is what justifies the differentiation performed below: we apply the annihilating operator $\frac{\dd}{\dd u}+\mu$ to an equation whose left-hand side involves $g'=\Phi''$, so that $g\in C^1$ is required.

Integrating by parts in~\eqref{eq:IDE} and making the change of variable $z=u-y$, we obtain an equation for the survival density $g(u):=\Phi'(u)$:
\begin{equation}
    \cL[g](u)=J(u),
    \qquad
    J(u):=\lambda e^{-\mu u}\Bigl(\Phi(0)+\int_0^ug(z)e^{\mu z}\,\dd z\Bigr),
    \label{eq:integral_g}
\end{equation}
where $\cL$ is the first-order linear differential operator
\begin{equation*}
    \cL[g](u):=\tfrac12\sigma_\kappa^2u^2g'(u)+(a_\kappa u+c)g(u).
\end{equation*}
Direct differentiation of $J$ gives $J'=-\mu J+\lambda g$, that is, $\bigl(\frac{\dd}{\dd u}+\mu\bigr)J=\lambda g$. Applying this annihilating operator to both sides of~\eqref{eq:integral_g} eliminates the integral nonlocality completely: $\frac{\dd}{\dd u}\cL[g]+\mu\cL[g]-\lambda g=0$. Expanding $\cL$ and grouping terms, we obtain the second-order linear homogeneous ODE
\begin{equation*}
    \tfrac12\sigma_\kappa^2u^2g''
    +\Bigl[\tfrac12\mu\sigma_\kappa^2u^2+(\sigma_\kappa^2+a_\kappa)u+c\Bigr]g'
    +\bigl[\mu a_\kappa u+(a_\kappa+\mu c-\lambda)\bigr]g=0 ,
\end{equation*}
and, after multiplication by $2/\sigma_\kappa^{2}$ and use of~\eqref{eq:structural},
\begin{equation}
    u^2g''(u)+\bigl[\mu u^2+(\gamma+2)u+\beta\bigr]g'(u)
    +\bigl[\mu\gamma u+(\mu\beta+\gamma-\nu)\bigr]g(u)=0 .
    \label{eq:ODE_final}
\end{equation}
Equation~\eqref{eq:ODE_final} is the central analytical object of the paper. For $\kappa=1$, $r=0$ it is equation~(13) of~\cite{Frolova2002} and the equation of~\cite[\S6]{KabPerg2016}. We reproduce the derivation only to fix the notation.

This section relies on Theorem~1 of~\cite{P2026} concerning the survival probability $\Phi$ and serves to derive equation~\eqref{eq:ODE_final}, that is, the implication $\Phi\Rightarrow{}$ODE. We use the existence of a classical solution in an essential way here, and nowhere else in the paper. In Section~\ref{sec:properties} and in the verification theorem of Section~\ref{sec:verification}, neither Theorem~1 of~\cite{P2026}, nor Lemma~\ref{lem:bootstrap}, nor any other property of $\Phi$ is used: we derive all properties of the constructed function $\wt\Phi$ from~\eqref{eq:ODE_final}, and the converse implication ODE${}\Rightarrow{}$IDE needed there is obtained independently in Proposition~\ref{prop:equivalence}. Thus the main result, Theorem~\ref{thm:exact_solution}, does not depend on the a priori regularity of $\Phi$. On the contrary, this regularity follows from it: $\Phi=\wt\Phi\in C^{2}([0,\infty))$ and $\Phi$ is real analytic on $(0,\infty)$.

\begin{remark}
\label{rem:general_F}
The elimination of the integral term rests on the existence of a differential operator with constant coefficients annihilating the integral operator, and such an operator exists precisely when the claim size distribution has a rational Laplace transform. Besides the exponential distribution, this covers the Erlang distributions $\mathrm{Erlang}(n,\mu)$, for which the annihilator is $\bigl(\frac{\dd}{\dd u}+\mu\bigr)^{n}$, mixtures of $n$ exponentials, for which it equals $\prod_{i=1}^{n}\bigl(\frac{\dd}{\dd u}+\mu_i\bigr)$, and, more generally, phase-type distributions. In these cases the reduction leads to a higher-order ODE, as in the series of works of Gerber and Shiu~\cite{GerberShiu1998,GerberShiu2005}. For a two-sided combination of exponentials the order rises to three~\cite[\S7]{KPuk}. For an arbitrary claim size distribution no finite annihilator exists. This approach thereby complements this literature: it gives a closed-form solution in the simplest, exponential, case, where the reduced equation is of second order and we identify it with a Heun equation (Section~\ref{sec:classification}).
\end{remark}

\section{Singular points and the doubly confluent Heun equation}
\label{sec:classification}

Written with leading coefficient one, equation~\eqref{eq:ODE_final} takes the form $g''+p(u)g'+q(u)g=0$ with rational coefficients
\begin{equation}
    p(u)=\mu+\frac{\gamma+2}{u}+\frac{\beta}{u^{2}},\qquad
    q(u)=\frac{\mu\gamma}{u}+\frac{\mu\beta+\gamma-\nu}{u^{2}} .
    \label{eq:pq}
\end{equation}
Therefore the only possible singular points on $\overline{\bbc}$ are $u=0$ and $u=\infty$.

As $u\to0$ we have $p(u)\sim\beta/u^{2}$ and $q(u)=O(u^{-2})$. Since $\beta>0$, the pole of $p$ has order exactly two, the Fuchs condition is violated, and $u=0$ is an irregular singular point. The Poincar\'e rank $r_0$ is the smallest integer $r\ge0$ for which the functions $u^{r+1}p(u)$ and $u^{2r+2}q(u)$ are both holomorphic at zero. From $\ord(p,0)=2$ it follows that $r\ge1$, whereas the condition on $q$ holds for $r\ge0$, whence $r_0=1$.

Under the involution $t=1/u$, $\wt g(t)=g(1/t)$, the coefficients take the form
\begin{equation*}
    \wt p(t)=\frac{2}{t}-\frac{1}{t^{2}}p\Bigl(\frac1t\Bigr)
            =-\frac{\mu}{t^{2}}-\frac{\gamma}{t}-\beta,
    \qquad
    \wt q(t)=\frac{1}{t^{4}}q\Bigl(\frac1t\Bigr)
            =\frac{\mu\gamma}{t^{3}}+\frac{\mu\beta+\gamma-\nu}{t^{2}} .
\end{equation*}
As $t\to0$ the function $\wt p$ has a pole of the second order and $\wt q$ one of the third order, so that $u=\infty$ is irregular as well. The requirement that $t^{r+1}\wt p$ and $t^{2r+2}\wt q$ be holomorphic gives $r+1\ge2$ and $2r+2\ge3$, whence $r_\infty=1$. That the rank is the integer $1$ rather than the half-integer $1/2$ (a half-integer rank corresponds to a ramified singular point, in which the solutions contain $\sqrt{u}$) follows from the characteristic equation for the exponential solutions $e^{\varrho u}$ of equation~\eqref{eq:ODE_final}, namely $\varrho^{2}+\mu\varrho=0$, whose roots $\varrho\in\{0,-\mu\}$ are distinct. A ramified irregular point would require a multiple root.

According to the classification of second-order linear equations by the $s$-rank of their singular points~\cite[Chapter~1]{Slavyanov2002}, an equation on $\overline{\bbc}$ with two irregular singular points of rank one belongs to the class of doubly confluent Heun equations. We make this identification explicit.

Since, to the best of our knowledge, these equations do not appear explicitly in the actuarial literature, we recall them here. The general Heun equation is a second-order linear equation with four regular singular points $0$, $1$, $a$, $\infty$:
\begin{equation}
    \frac{\dd^{2}w}{\dd z^{2}}
    +\Bigl(\frac{\gamma_H}{z}+\frac{\delta_H}{z-1}+\frac{\eps_H}{z-a}\Bigr)\frac{\dd w}{\dd z}
    +\frac{\alpha_H\beta_H z-q_H}{z(z-1)(z-a)}\,w=0 ,
    \label{eq:heun_general}
\end{equation}
with the parameter constraint $\alpha_H+\beta_H+1=\gamma_H+\delta_H+\eps_H$. See~\cite[Chapter~31]{DLMF} and~\cite{Ronveaux}. It occupies in the hierarchy of linear equations the step following the hypergeometric one: the hypergeometric equation has three regular singular points, and its local data (the positions of the points and the exponents) determine the equation uniquely, whereas in~\eqref{eq:heun_general} there remains one more free parameter $q_H$, connected neither with the positions of the singular points nor with the exponents at them. Such a parameter is called accessory.

The presence of an accessory parameter changes the quality of the theory. The coefficients of a local power series satisfy not a two-term but a three-term recurrence relation, and therefore they are not expressible as a ratio of gamma functions. Integral representations of Euler type, which for the hypergeometric function yield the Gauss connection formulas, are in general absent here. As a consequence, the central connection problem (expressing a solution singled out by a condition at one singular point through solutions singled out by conditions at another) is unsolved for equations of the Heun class in general. The monograph~\cite{Lay} is devoted to it. This is where the difficulty with the constant $\cK$ lies, and we return to it in Section~\ref{sec:laplace}.

The coalescence of singular points of equation~\eqref{eq:heun_general} produces four confluent forms~\cite[\S31.12]{DLMF}: the \textit{confluent} Heun equation (two regular singular points and one irregular point of rank~$1$), the \textit{doubly confluent} (two irregular points of rank~$1$), the \textit{biconfluent} (one regular point and one irregular point of rank~$2$), and the \textit{triconfluent} (a single irregular point of rank~$3$). Each such coalescence turns a pair of regular singular points into a single irregular one, and along with it divergent formal solutions appear: at an irregular singular point the Thom\'e series have, generically, zero radius of convergence and are to be understood asymptotically. Equation~\eqref{eq:ODE_final} belongs, as shown above, to the doubly confluent case, and this case entails the two distinctive features of the subsequent exposition: the solution at the origin has to be specified not by the sum of a series but by a boundedness condition (Lemma~\ref{lem:local0}), and the constant in the power-law tail is a connection coefficient (Theorem~\ref{thm:connection}).

\begin{proposition}[Canonical form]
\label{prop:DCHE}
Let $z=\mu u$ and $w(z)=g(z/\mu)$. Then~\eqref{eq:ODE_final} takes the form
\begin{equation}
    z^{2}w''(z)+\bigl[z^{2}+(\gamma+2)z+\mu\beta\bigr]w'(z)
    +\bigl[\gamma z+(\mu\beta+\gamma-\nu)\bigr]w(z)=0 ,
    \label{eq:DCHE}
\end{equation}
which is exactly the doubly confluent Heun equation in the canonical form~\cite[(31.12.2)]{DLMF}
\begin{equation}
    \frac{\dd^{2}w}{\dd z^{2}}
    +\Bigl(\frac{\delta}{z^{2}}+\frac{\gamma_{\mathrm D}}{z}+1\Bigr)\frac{\dd w}{\dd z}
    +\frac{\alpha_{\mathrm D}z-q}{z^{2}}\,w=0
    \label{eq:DLMF}
\end{equation}
with the parameter dictionary
\begin{equation}
    \delta=\mu\beta,\qquad
    \gamma_{\mathrm D}=\gamma+2,\qquad
    \alpha_{\mathrm D}=\gamma,\qquad
    q=\nu-\gamma-\mu\beta .
    \label{eq:dictionary}
\end{equation}
All four parameters are dimensionless.
\end{proposition}

\begin{proof}
For $z=\mu u$ we have $\dd/\dd u=\mu\,\dd/\dd z$, so that $u^{2}g''=z^{2}w''$, $\bigl[\mu u^{2}+(\gamma+2)u+\beta\bigr]g'=\bigl[z^{2}+(\gamma+2)z+\mu\beta\bigr]w'$, and $\bigl[\mu\gamma u+(\mu\beta+\gamma-\nu)\bigr]g =\bigl[\gamma z+(\mu\beta+\gamma-\nu)\bigr]w$. Dividing~\eqref{eq:DCHE} by $z^{2}$ and comparing with~\eqref{eq:DLMF}, we obtain~\eqref{eq:dictionary}.
\end{proof}

One could attempt to write the solution constructed below in terms of the function denoted $\mathrm{HeunD}$. This should not be done, and it is appropriate to record the reason, since it is not specific to this paper.

The notations $\mathrm{HeunG}$, $\mathrm{HeunC}$, $\mathrm{HeunD}$, $\mathrm{HeunB}$, $\mathrm{HeunT}$ for the general Heun equation and its four confluent forms are a convention of computer algebra systems (for instance, Maple or Wolfram), not classical notation (for example, in~\cite{Lay,Ronveaux,Slavyanov2002} such notations are not used). This convention is moreover not unified. In Maple the function \texttt{HeunD} has four parameters, and its canonical equation carries irregular singular points of rank one at $z=\pm1$ together with the ordinary point $z=0$, at which the function is normalized by the conditions $y(0)=1$, $y'(0)=0$, the series converging for $|z|<1$. In the Wolfram Language the function \texttt{HeunD} has five parameters and is normalized at a different point.

In none of the systems is the function normalized at an irregular singular point; this is impossible in principle, since at such a point the power series diverges, and the condition $y(0)=1$, $y'(0)=0$ would single out a solution only formally. The branch we need is singled out precisely by boundedness at the irregular origin, which is inconvenient to express through the notations of computer algebra systems. We therefore do not use the names of computer algebra systems, we give the canonical form according to~\cite[(31.12.2)]{DLMF}, and we define the required solution intrinsically (Definition~\ref{def:H} and Lemma~\ref{lem:local0}), accompanying it with the three-term recurrence relation~\eqref{eq:recurrence} for the coefficients of the asymptotic series, which serves as the working computational tool in Section~\ref{sec:numerics}.

\section{Analysis of the ODE solution}
\label{sec:properties}

\begin{definition}
\label{def:H}
By $\cH$ we denote the solution of equation~\eqref{eq:ODE_final} on $(0,\infty)$ which is bounded as $u\to0+$, extended to the origin by continuity, and normalized by the condition $\cH(0)=1$.
\end{definition}
Lemma~\ref{lem:local0} shows that $\cH$ exists, is unique, and belongs to $C^{2}([0,\infty))$. The equation itself is not posed at the point $u=0$: the leading coefficient $u^{2}$ vanishes there, and the origin is an irregular singular point.

\begin{lemma}[Local analysis at the irregular origin]
\label{lem:local0}
Let $\gamma>1$, $\beta>0$, $\nu>0$. Then:
\begin{enumerate}
\item[\rm(i)] equation~\eqref{eq:ODE_final} has at the point $u=0$ a fundamental system
consisting of one solution bounded at the origin and one solution
$\wt g(u)=e^{\beta/u}u^{-\gamma}\bigl(1+O(u)\bigr)$ as $u\to0+$;
\item[\rm(ii)] the bounded solution is unique up to a multiplicative constant. Normalized
by the condition $\cH(0)=1$, it satisfies $\cH\in C^{2}([0,\infty))\cap C^{\omega}((0,\infty))$,
\begin{equation}
    \cH'(0)=-\frac{\mu\beta+\gamma-\nu}{\beta},
    \qquad
    u^{2}\cH'(u)\xrightarrow[u\to0+]{}0 ,
    \label{eq:H0}
\end{equation}
and is asymptotic, in the Poincar\'e sense as $u\to0+$, to the formal series $\sum_{n\ge0}a_nu^{n}$
determined by the condition $a_0=1$ and the three-term recurrence
\begin{equation}
    \beta(n+1)a_{n+1}
    =-\bigl[n(n+\gamma+1)+\mu\beta+\gamma-\nu\bigr]a_n
     -\mu(n-1+\gamma)a_{n-1},
    \quad a_{-1}=0 ;
    \label{eq:recurrence}
\end{equation}
in particular $\cH''(0)=2a_2$;
\item[\rm(iii)] the coefficients of the series in~\rm(ii) satisfy the bound $|a_n|\le AB^{n}n!$
with some constants $A,B$, that is, the series is of Gevrey class~$1$. Generically it
diverges. Namely, the three-term recurrence~\eqref{eq:recurrence} has a dominant
solution with $|a_{n+1}/a_n|\sim n/\beta$ and a recessive one with $a_n=O(\mu^{n}/n!)$, more precisely $a_n\asymp(-\mu)^{n}n^{-2}/n!$, and the initial
data $a_{-1}=0$, $a_0=1$ correspond to the recessive one only for exceptional parameter
values: convergence of the series is equivalent to the vanishing of the Stokes multiplier,
an analytic function of the triple $(\gamma,\mu\beta,\nu)$, so that the exceptional set has
codimension at least one. Independently
of this, the series is $1$-summable, that is, Borel summable: the Borel transform
$\sum_na_n\zeta^{n}/n!$ converges in a neighborhood of zero and continues analytically with
an exponential growth bound along every ray $\arg\zeta=\vartheta$ except the single
Stokes direction $\vartheta=\pi$, and the Laplace integral along such a ray returns the
unique function asymptotic to the series in a sector of opening larger than~$\pi$
(see~\cite[Chapters~3 and~8]{Balser}). For $\vartheta=0$ this function is exactly $\cH$: it
is recovered from its series uniquely.
\end{enumerate}
\end{lemma}

\begin{proof}
We give the proof in Appendix~\ref{app:local0}.
\end{proof}

None of the parameter sets of Section~\ref{sec:numerics} falls into the exceptional set of item~(iii): there the ratio $|a_{n+1}/a_n|$ reproduces $n/\beta$ to within one percent already at $n\approx200$, that is, the divergent branch is realized.

The next lemma is the technical center of the section. It is elementary, but everything in this section rests on it.

\begin{lemma}[First integral]
\label{lem:firstintegral}
Let $g$ be any solution of~\eqref{eq:ODE_final} on $(0,\infty)$ having a finite limit $g(0+)$ and such that $u^{2}g'(u)\to0$ as $u\to0+$. Then for all $u\ge0$
\begin{equation}
    \cL[g](u)=c\,g(0+)\,e^{-\mu u}
             +\lambda e^{-\mu u}\int_0^{u}g(z)e^{\mu z}\,\dd z .
    \label{eq:firstintegral}
\end{equation}
\end{lemma}

\begin{proof}
Set
\[D(u):=e^{\mu u}\cL[g](u)-\lambda\int_0^{u}g(z)e^{\mu z}\dd z.\]
Then $D'=e^{\mu u}\bigl(\mu\cL[g]+\frac{\dd}{\dd u}\cL[g]-\lambda g\bigr)$, and the bracket vanishes identically, since $g$ solves~\eqref{eq:ODE_final}, which is exactly what $\frac{\dd}{\dd u}\cL[g]+\mu\cL[g]-\lambda g=0$ means. Hence $D$ is constant. Its value at the origin is $D(0+)=\cL[g](0+)=\lim_{u\to0+}\bigl[\tfrac12\sigma_\kappa^2u^2g'(u) +(a_\kappa u+c)g(u)\bigr]=c\,g(0+)$ by the assumption. Multiplying the identity $D\equiv cg(0+)$ by $e^{-\mu u}$, we obtain~\eqref{eq:firstintegral}.
\end{proof}

\begin{proposition}[Equivalence of the ODE and the IDE]
\label{prop:equivalence}
Let $g$ be as in Lemma~\ref{lem:firstintegral}, let $\Phi_0\in\bbr$, and let $\Phi(u):=\Phi_0+\int_0^{u}g$. Then the residual of the IDE equals
\begin{equation}
    \cL[g](u)-J(u)=\bigl(c\,g(0+)-\lambda\Phi_0\bigr)e^{-\mu u},
    \qquad u\ge0 ,
    \label{eq:residual}
\end{equation}
where $J$ is defined in~\eqref{eq:integral_g}. Moreover, $\Phi$ solves~\eqref{eq:IDE} on $(0,\infty)$ if and only if
\begin{equation}
    c\,g(0+)=\lambda\Phi_0,
    \qquad\text{that is}\qquad
    c\,\Phi'(0)=\lambda\Phi(0).
    \label{eq:boundary}
\end{equation}
\end{proposition}

\begin{proof}
Subtract $J(u)=\lambda e^{-\mu u}\bigl(\Phi_0+\int_0^{u}ge^{\mu z}\bigr)$ from~\eqref{eq:firstintegral}. The integral terms cancel and~\eqref{eq:residual} remains. The residual vanishes identically if and only if its multiplier is zero, and~\eqref{eq:integral_g} is equivalent to~\eqref{eq:IDE} for every $\Phi\in C^{1}([0,\infty))\cap C^{2}((0,\infty))$: the passage between them is an integration by parts in $\int_0^u\Phi(u-y)\,\dd F(y)$, legitimate by the absolute continuity of $\Phi$ on each interval $[0,u]$. The function $\Phi=\Phi_0+\int_0^{u}g$ possesses this regularity: $g$ is continuous on $(0,\infty)$ and has a finite limit at zero, hence is bounded and integrable near zero.
\end{proof}

Proposition~\ref{prop:equivalence} gives the implication ODE${}\Rightarrow{}$IDE, which does not follow from the derivation in Section~\ref{sec:reduction}: applying the annihilator enlarges the solution set. It also clarifies the status of the boundary relation~\eqref{eq:boundary}: this is not an additional condition but precisely the condition under which a solution of the second-order equation solves the original first-order relation. The number of conditions is thereby consistent: the solution space of~\eqref{eq:ODE_final} is two-dimensional, boundedness at zero lowers it to one dimension by Lemma~\ref{lem:local0}, and~\eqref{eq:boundary} is one linear scalar relation between $g(0+)$ and the integration constant. It is homogeneous in the pair $(\Phi_0,g)$ and therefore fixes it only up to a common factor, which is removed by the normalization~\eqref{eq:Cdef}.

\begin{lemma}[Positivity]
\label{lem:positivity}
$\cH(u)>0$ for all $u\ge0$.
\end{lemma}

\begin{proof}
By Lemma~\ref{lem:local0} the function $\cH$ satisfies the assumptions of Lemma~\ref{lem:firstintegral} with $\cH(0+)=1$, hence
\begin{equation}
    \tfrac12\sigma_\kappa^2u^2\cH'(u)+(a_\kappa u+c)\cH(u)
    =c\,e^{-\mu u}+\lambda e^{-\mu u}\int_0^{u}\cH(z)e^{\mu z}\,\dd z .
    \label{eq:HFI}
\end{equation}
Suppose that $\cH$ vanishes somewhere on $(0,\infty)$, and set $u_0:=\inf\{u>0:\cH(u)=0\}$. Since $\cH(0)=1$ and $\cH$ is continuous on $[0,\infty)$, we have $u_0>0$, $\cH>0$ on $[0,u_0)$, and $\cH(u_0)=0$. Evaluating~\eqref{eq:HFI} at $u_0$ and using $\cH(u_0)=0$, we obtain
\begin{equation*}
    \tfrac12\sigma_\kappa^2u_0^{2}\,\cH'(u_0)
    =c\,e^{-\mu u_0}+\lambda e^{-\mu u_0}\int_0^{u_0}\cH(z)e^{\mu z}\,\dd z>0,
\end{equation*}
since $c>0$ and the integrand is strictly positive on $(0,u_0)$. Hence $\cH'(u_0)>0$. On the other hand, from $\cH>0$ on $[0,u_0)$ and $\cH(u_0)=0$ it follows that $\cH'(u_0)=\lim_{u\uparrow u_0}\frac{\cH(u_0)-\cH(u)}{u_0-u} =\lim_{u\uparrow u_0}\frac{-\cH(u)}{u_0-u}\le0$, a contradiction. Consequently, $\cH$ has no zeros on $(0,\infty)$, and being positive at the origin and continuous, $\cH>0$ everywhere.
\end{proof}

\begin{remark}
The proof uses only the sign of the derivative $\cH'$ at the first zero of $\cH$. There is no monotonicity of $\cH$, and the proof uses none: by~\eqref{eq:H0} the sign of $\cH'(0)$ is the sign of $\nu-\gamma-\mu\beta$, so that whenever $\nu>\gamma+\mu\beta$ (possible only under a negative safety loading) the density first increases and only then decreases.
\end{remark}

\begin{lemma}[Behavior at infinity]
\label{lem:infty}
As $u\to+\infty$ equation~\eqref{eq:ODE_final} has a fundamental system
\begin{equation}
    g_1(u)=(\mu u)^{-\gamma}\bigl(1+O(u^{-1})\bigr),
    \qquad
    g_2(u)=e^{-\mu u}(\mu u)^{-2}\bigl(1+O(u^{-1})\bigr),
    \label{eq:fsinf}
\end{equation}
and there exist constants $K_1,K_2$ such that
\begin{equation}
    \cH(u)=K_1g_1(u)+K_2g_2(u).
    \label{eq:connection}
\end{equation}
For $\gamma>1$ both modes are integrable on $(1,\infty)$, so that $\int_0^\infty\cH<\infty$. The expansions~\eqref{eq:fsinf} can be differentiated term by term.
\end{lemma}

\begin{proof}
We give the proof in Appendix~\ref{app:infty}.
\end{proof}

\begin{corollary}
\label{cor:admissible}
Let $\gamma>1$. Set
\begin{equation}
    C:=\Bigl(\frac{c}{\lambda}+\int_0^{\infty}\cH(v)\,\dd v\Bigr)^{-1},
    \qquad
    \wt\Phi(u):=C\Bigl(\frac{c}{\lambda}+\int_0^{u}\cH(v)\,\dd v\Bigr),\quad u\ge0 .
    \label{eq:Cdef}
\end{equation}
Then $0<C<\infty$ and $\wt\Phi$ is strictly increasing on $[0,\infty)$, $\wt\Phi(0)=Cc/\lambda\in(0,1)$, $\lim_{u\to\infty}\wt\Phi(u)=1$, and $0<\wt\Phi(u)\le1$ for all $u\ge0$. Moreover, $\wt\Phi$ satisfies~\eqref{eq:IDE} on $(0,\infty)$ and belongs to $C^{2}([0,\infty))$, with $\wt\Phi'=C\cH$ and $\wt\Phi''=C\cH'$.
\end{corollary}

\begin{proof}
By Lemma~\ref{lem:positivity} we have $\cH>0$, and by Lemma~\ref{lem:infty}, for $\gamma>1$, $\int_0^{\infty}\cH<\infty$. Hence $0<\int_0^\infty\cH<\infty$, and $c/\lambda>0$, so the bracket in~\eqref{eq:Cdef} is finite and strictly positive and $C\in(0,\infty)$. Then $\wt\Phi'=C\cH>0$ gives strict monotonicity, $\wt\Phi(\infty)=C(c/\lambda+\int_0^\infty\cH)=1$ by construction, and $\wt\Phi(0)=Cc/\lambda$ lies strictly between $0$ and $\wt\Phi(\infty)=1$. Monotonicity together with the two boundary values gives $0<\wt\Phi\le1$. Further, $\wt\Phi'(0)=C\cH(0)=C$ and $\wt\Phi(0)=(c/\lambda)C$, so that~\eqref{eq:boundary} holds and Proposition~\ref{prop:equivalence} applies. Finally, $\cH\in C^{1}([0,\infty))$ by Lemma~\ref{lem:local0}\,(ii), and $\wt\Phi'=C\cH$ and $\wt\Phi''=C\cH'$, whence $\wt\Phi\in C^{2}([0,\infty))$.
\end{proof}

Corollary~\ref{cor:admissible} is what makes the verification theorem of Section~\ref{sec:verification} self-contained: the bounds $0<\wt\Phi\le1$, needed there to promote a local martingale to a uniformly integrable one, are now consequences of the equation, not of the probabilistic interpretation of $\wt\Phi$.

\begin{lemma}[An integro-differential equation for the ruin probability]
\label{lem:idePsi}
The function $\wt\Psi:=1-\wt\Phi$ satisfies, for $u>0$, the equation
\begin{equation}
    \tfrac12\sigma_\kappa^{2}u^{2}\wt\Psi''+(a_\kappa u+c)\wt\Psi'-\lambda\wt\Psi
    +\lambda\mu e^{-\mu u}\!\int_0^{u}\!\wt\Psi(z)e^{\mu z}\dd z
    =-\lambda e^{-\mu u}.
    \label{eq:IDEPsi}
\end{equation}
\end{lemma}

\begin{proof}
Substituting the constant function $1$ into the operator~\eqref{eq:generator} gives $\cA1=-\lambda+\lambda F(u)=-\lambda e^{-\mu u}$. By Corollary~\ref{cor:admissible}, $\cA\wt\Phi\equiv0$ on $(0,\infty)$, and $\cA$ is linear, so $\cA\wt\Psi=\cA1-\cA\wt\Phi=-\lambda e^{-\mu u}$, which is~\eqref{eq:IDEPsi}.
\end{proof}

\begin{lemma}[Sign of the connection coefficient]
\label{lem:K1pos}
$K_1>0$.
\end{lemma}

\begin{proof}
Suppose that $K_1=0$. Then $\cH=K_2g_2$, and $\cH>0$ implies $K_2>0$. Since $\cH(z)e^{\mu z}=O(z^{-2})$, the integral $I_\infty:=\int_0^{\infty}\cH(z)e^{\mu z}\dd z$ is finite, and the right-hand side of~\eqref{eq:HFI} multiplied by $e^{\mu u}$ tends to $c+\lambda I_\infty>0$. In the left-hand side, $\cH=K_2g_2$ gives $\cH'=-\mu\cH(1+o(1))$ and $u^{2}\cH=K_2\mu^{-2}e^{-\mu u}(1+o(1))$, hence
\begin{equation*}
    \tfrac12\sigma_\kappa^{2}u^{2}\cH'
    =-\tfrac12\sigma_\kappa^{2}\mu\,u^{2}\cH\,(1+o(1))
    =-\frac{\sigma_\kappa^{2}}{2\mu}\,K_2\,e^{-\mu u}\bigl(1+o(1)\bigr),
\end{equation*}
which dominates $(a_\kappa u+c)\cH=O(u^{-1}e^{-\mu u})$. The left-hand side of~\eqref{eq:HFI} multiplied by $e^{\mu u}$ therefore tends to $-\sigma_\kappa^{2}K_2/(2\mu)<0$, whereas the right-hand side tends to a positive limit. The resulting contradiction gives $K_1\neq0$. For $\gamma>1$ the mode $g_1$ dominates $g_2$ at infinity, so that $\cH(u)\sim K_1(\mu u)^{-\gamma}$, and $\cH>0$ forces $K_1>0$.
\end{proof}

Of the two constants in~\eqref{eq:connection} only $K_1$ is uniquely determined. Indeed,
adding to the dominant mode $g_1$ any multiple of the recessive mode $g_2$ again gives a
solution with the same power-law behavior at infinity, so the expansion~\eqref{eq:fsinf}
fixes $g_1$ only up to such a term. Under this change $K_2$ varies, but $K_1$ does not.
Only $K_1$ matters below.

\section{Verification theorem}
\label{sec:verification}

We extend the function $\wt\Phi$ constructed in~\eqref{eq:Cdef} to the whole line by setting $\wt\Phi(u):=0$ for $u<0$. The extension affects only the strictly negative half-axis: at the point $u=0$ itself the value $\wt\Phi(0)=Cc/\lambda\in(0,1)$ is already given by formula~\eqref{eq:Cdef}, since $\cH$ is defined and continuous at the origin. It would be incorrect to set $\wt\Phi$ equal to zero at $u=0$: ruin is by definition the exit of the reserve strictly below zero, the survival probability at zero initial capital is positive, and the value of the ruin probability $\Psi(0)=1-\Phi(0)$ is studied below. We place this section before Section~\ref{sec:solution} deliberately: it uses neither Lemma~\ref{lem:bootstrap}, nor the results of~\cite{P2026}, nor any property of the survival probability $\Phi$, so that no logical circle arises in the derivation of the explicit formula for $\Phi$. We use only Section~\ref{sec:properties}, the pathwise properties of the reserve from Lemma~\ref{lem:origin}, the transience $X_t\to+\infty$ on $\{\tau=\infty\}$ for $\gamma>1$~\cite{Paulsen1993}, and standard semimartingale tools~\cite{Protter2005}.

\begin{theorem}
\label{thm:verification}
Let $\gamma>1$. Then for every $u\ge0$
\begin{equation*}
    \wt\Phi(u)=\bbp\bigl(\tau=\infty\mid X_0=u\bigr).
\end{equation*}
\end{theorem}

\begin{proof}
By Lemma~\ref{lem:local0}\,(ii) we have $\cH\in C^{1}([0,\infty))$, and $\wt\Phi'=C\cH$ by~\eqref{eq:Cdef}, whence $\wt\Phi\in C^{2}([0,\infty))$, the derivatives at the origin being understood as one-sided. Moreover, the coefficients~\eqref{eq:pq} are real analytic on $(0,\infty)$, and the leading coefficient $u^{2}$ of equation~\eqref{eq:ODE_final} does not vanish there, so that $\cH$, and with it $\wt\Phi$, is real analytic on $(0,\infty)$. By Corollary~\ref{cor:admissible} the function $\wt\Phi$ is bounded, with $0<\wt\Phi\le1$ on $[0,\infty)$, strictly increasing, satisfies~\eqref{eq:IDE} on $(0,\infty)$, and has $\wt\Phi(\infty)=1$. All these properties are consequences of equation~\eqref{eq:ODE_final} and of Section~\ref{sec:properties}. None of them is borrowed from $\Phi$, and the results of~\cite{P2026} are not used in this proof.

Fix $X_0=u>0$. We first discuss the regularity required for the It\^o formula: the function $\wt\Phi$ belongs to $C^{2}([0,\infty))$ but has been extended by zero to the negative half-axis and is therefore discontinuous at the origin: $\wt\Phi(0-)=0$, whereas $\wt\Phi(0)=\wt\Phi(0+)=Cc/\lambda>0$. In particular, the extended function is right-continuous everywhere. The extension does not affect the differential part of the formula: it enters only the terms responsible for jumps (the integral term of the generator~\eqref{eq:generator} and the compensated integral in~\eqref{eq:Ito}, where $\wt\Phi$ is evaluated at the points $X_{s-}-y$), as well as the terminal value $\wt\Phi(X_{t\wedge\tau\wedge\tau_\eps})$. In all these places only measurability and boundedness of $\wt\Phi$, established in Corollary~\ref{cor:admissible}, are required. Since the jumps of $X$ have finite activity, on $[0,t\wedge\tau\wedge\tau_\eps]$ there are a.s.\ finitely many of them, and between consecutive jumps $X$ is a continuous semimartingale with values in $[\eps,\infty)$, where $\wt\Phi$ is of class $C^{2}$. Applying It\^o's formula on each inter-jump interval and adding the jump increments $\wt\Phi(X_{s-}-y)-\wt\Phi(X_{s-})$, written as an integral against the jump measure, we obtain~\eqref{eq:Ito}. The compensation step uses only the boundedness and measurability of $\wt\Phi$ on $\bbr$. The set of those $y$ for which $X_{s-}-y=0$ is a single point, and the claim distribution is continuous, so this set has probability zero. By Lemma~\ref{lem:origin} the reserve is strictly positive on $[0,\tau)$ and satisfies $X_\tau<0$ a.s.\ on $\{\tau<\infty\}$. Therefore, for $0<\eps<u$ set $\tau_\eps:=\inf\{t>0:X_t\le\eps\}$ and apply It\^o's formula for semimartingales with jumps~\cite[Chapter~II, Theorem~33]{Protter2005} on $[0,t\wedge\tau\wedge\tau_\eps]$, where the differential part acts on the restriction of $\wt\Phi$ to $[\eps,\infty)$, on which it belongs to $C^{2}$:
\begin{multline}
    \wt\Phi(X_{t\wedge\tau\wedge\tau_\eps})=\wt\Phi(u)
      +\int_0^{t\wedge\tau\wedge\tau_\eps}\cA\wt\Phi(X_{s-})\,\dd s
      +\int_0^{t\wedge\tau\wedge\tau_\eps}\sigma_\kappa X_{s-}\wt\Phi'(X_{s-})\,\dd W_s + \\
      +\int_0^{t\wedge\tau\wedge\tau_\eps}\!\!\int_0^{\infty}
        \bigl(\wt\Phi(X_{s-}-y)-\wt\Phi(X_{s-})\bigr)\,\wt N(\dd s,\dd y),
    \label{eq:Ito}
\end{multline}
where $\wt N$ is the compensated Poisson random measure and $\cA$ is the operator
\eqref{eq:generator}, in whose integral term the extension $\wt\Phi\equiv0$ to the
negative half-axis correctly accounts for the jumps leading to ruin. Since $\wt\Phi$ solves~\eqref{eq:IDE}, we have $\cA\wt\Phi\equiv0$ on $(0,\infty)$, and the time integral in~\eqref{eq:Ito} vanishes. Therefore the stopped process $Y^{\eps}_t:=\wt\Phi(X_{t\wedge\tau\wedge\tau_\eps})$ is a local martingale, and being bounded by Corollary~\ref{cor:admissible}, it is a uniformly integrable martingale~\cite[Chapter~I, Theorem~51]{Protter2005}, whence $\bbe_u\bigl[\wt\Phi(X_{t\wedge\tau\wedge\tau_\eps})\bigr]=\wt\Phi(u)$ for all $t\ge0$. Fix $t$. By Lemma~\ref{lem:origin} the continuous component of the reserve is strictly positive, there are finitely many jumps on $[0,t]$, and the probability that a jump lands exactly at zero is zero, so $\inf_{[0,t\wedge\tau)}X>0$ a.s. Hence for all sufficiently small, in general random, $\eps$ we have $t\wedge\tau\wedge\tau_\eps=t\wedge\tau$, and the limiting procedure yields exactly $X_{t\wedge\tau}$ and not the left limit. Since $\wt\Phi$ is bounded, the dominated convergence theorem allows passing to the limit first in $\eps$ and then in $t$, which gives $\wt\Phi(u)=\bbe_u\bigl[\lim_{t\to\infty}\wt\Phi(X_{t\wedge\tau})\bigr]$.

It remains to identify $\eta:=\lim_{t\to\infty}\wt\Phi(X_{t\wedge\tau})$. On $\{\tau<\infty\}$, Lemma~\ref{lem:origin} gives $X_\tau<0$ a.s.\ (and $\bbp(X_\tau=0)=0$), so that $\wt\Phi(X_\tau)=0$. After $\tau$ the stopped process remains at $X_\tau$, whence $\eta=0$. On $\{\tau=\infty\}$ the trajectory remains strictly positive, and the condition $\gamma>1$ gives $X_t\to+\infty$ as $t\to\infty$ a.s.~\cite{Paulsen1993}, whence $\eta=\lim_{x\to\infty}\wt\Phi(x)=1$. Indeed, write $X_t=\mathcal E_tL_t$ with $L_t=u+c\int_0^{t}\mathcal E_s^{-1}\dd s-\sum_{T_k\le t}\xi_k\mathcal E_{T_k}^{-1}$, where $T_k$ are the claim arrival times. Since $t^{-1}\log\mathcal E_t\to a_\kappa-\tfrac12\sigma_\kappa^{2}>0$ a.s.\ and $T_k/k\to1/\lambda$, both the integral and the series converge a.s., so $L_t\to L_\infty$ and $\mathcal E_t\to\infty$. Conditionally on $W$ and on the arrival times, $L_\infty$ is an affine function of the independent $\mathrm{Exp}(\mu)$ variables $\xi_k$ with nonzero coefficients, hence has a continuous conditional law, so that $\bbp(L_\infty=0)=0$. On $\{\tau=\infty\}$ we have $L_\infty\ge0$, therefore $L_\infty>0$ and $X_t\to+\infty$ a.s. Thus $\eta=\Ind_{\{\tau=\infty\}}$ a.s.\ and $\wt\Phi(u)=\bbp(\tau=\infty\mid X_0=u)$. For $u=0$ the argument is carried out not from zero but from a time instant $t_0>0$: by Lemma~\ref{lem:origin}, on the event $\{\tau>t_0\}$ we have $X_{t_0}>0$ a.s., and the same computation on $[t_0,\,\cdot\,]$ gives what has been proved for the pair $(t_0,X_{t_0})$. The passage to the limit $t_0\downarrow0$ is legitimate, since $\bbp(\tau\le t_0)\le1-e^{-\lambda t_0}\to0$, the trajectories are right-continuous, and $\wt\Phi$ is bounded.
\end{proof}

Theorem~\ref{thm:verification} together with Proposition~\ref{prop:equivalence} also settles the question of uniqueness. Let $\gamma>1$, let $g$ be a solution of~\eqref{eq:ODE_final} bounded at the origin, and let $V(u):=V_0+\int_0^{u}g$. By Lemma~\ref{lem:local0} then $g=a\cH$ for some $a\in\bbr$, and the relation~\eqref{eq:boundary}, homogeneous in the pair $(V_0,a)$, gives only $V_0=ac/\lambda$. Hence there remains a one-parameter family $V_a(u)=a\bigl(c/\lambda+\int_0^{u}\cH\bigr)$, every member of which solves~\eqref{eq:IDE}. The normalization $\lim_{u\to\infty}V(u)=1$, legitimate by Lemma~\ref{lem:infty}, fixes $a=C$ from~\eqref{eq:Cdef}, that is, $V=\wt\Phi$. By Theorem~\ref{thm:verification} this unique function is the survival probability. Sufficient conditions of this type for insurance models with investments were obtained by Belkina~\cite{BelkinaSuff} and, for the model with stochastic premiums, in~\cite[Lemma~3]{BKK2012}.

\section{Exact solution}
\label{sec:solution}

\begin{theorem}
\label{thm:exact_solution}
Let $\gamma>1$, and let $\cH$ be as in Definition~\ref{def:H} and $C$ as in~\eqref{eq:Cdef}. Set
\begin{equation}
    \wt\Phi(u)=C\Bigl(\frac{c}{\lambda}+\int_0^{u}\cH(v)\,\dd v\Bigr),
    \qquad u\ge0 .
    \label{eq:Phi_explicit}
\end{equation}
Then:
\begin{enumerate}
\item[\rm(i)] $0<\wt\Phi\le1$ on $[0,\infty)$, the function $\wt\Phi$ is strictly increasing, and $\wt\Phi(\infty)=1$;
\item[\rm(ii)] $\wt\Phi$ solves~\eqref{eq:IDE} on $(0,\infty)$;
\item[\rm(iii)] $\wt\Phi=\Phi$, that is, formula~\eqref{eq:Phi_explicit} gives the survival probability for every initial capital $u\ge0$.
\end{enumerate}
\end{theorem}

\begin{proof}
Items~(i) and~(ii) constitute Corollary~\ref{cor:admissible}, which rests on Definition~\ref{def:H} and Lemmas~\ref{lem:local0}, \ref{lem:positivity}, and~\ref{lem:infty}. Item~(iii) is Theorem~\ref{thm:verification} of the preceding section.
\end{proof}

We compare Theorem~\ref{thm:exact_solution} with the closest explicit result. Paulsen and Gjessing~\cite{PaulsenGjessing1997} obtained closed expressions for the ruin probability (and for the Laplace transform of the time of ruin) in terms of confluent hypergeometric functions and modified Bessel functions for exponential claims, a mixture of two exponentials, and the Erlang distribution $\mathrm{Erlang}(2,\mu)$ (Examples~2.1--2.4). In all these examples, however, the volatility of the return is zero: the accumulation there is deterministic, and only the claims contribute randomness (in Example~2.2, also a diffusive perturbation of the reserve itself). The only case with volatile return in their paper is the purely diffusive limit (\S3, going back to~\cite{Paulsen1993}), in which the compound Poisson stream of claims is replaced by a Brownian component. The general equation~(2.11), written out by them for nonzero return volatility, remains unsolved. The survey~\cite[\S4]{Paulsen2008} states this as the general situation: for $\lambda>0$, in all known analytic solutions the return volatility vanishes. Our representation in terms of the solution $\cH$ covers simultaneously a volatile asset ($\sigma_\kappa>0$ for every fraction $\kappa\in(0,1]$, with a risk-free component) and a compound Poisson stream of claims. At the same time, we identify the equation itself with a doubly confluent Heun equation and derive the positivity, monotonicity, and boundedness of the constructed function from it, without recourse to its probabilistic meaning.

\begin{proposition}
\label{prop:ruin_asymptotics}
Under the conditions of Theorem~\ref{thm:exact_solution}, as $u\to\infty$,
\begin{equation}
    \Psi(u)\sim\cK\,u^{-(\gamma-1)},
    \qquad
    \cK=\frac{C\,K_1\,\mu^{-\gamma}}{\gamma-1}>0 ,
    \label{eq:tail}
\end{equation}
where $\gamma=2a_\kappa/\sigma_\kappa^{2}$ and $K_1$ is the connection coefficient from~\eqref{eq:connection}.
\end{proposition}

\begin{proof}
By Theorem~\ref{thm:exact_solution} we have $\Psi(u)=C\int_u^{\infty}\cH$. By Lemma~\ref{lem:infty} the leading term of $\cH$ at infinity is $K_1(\mu u)^{-\gamma}$, and $K_1>0$ by Lemma~\ref{lem:K1pos}. The exponentially decreasing mode contributes $O(e^{-\mu u})$. Therefore
\[\Psi(u)\sim CK_1\mu^{-\gamma}\int_u^{\infty}v^{-\gamma}\dd v =\frac{CK_1\mu^{-\gamma}}{\gamma-1}u^{-(\gamma-1)}.\]
\end{proof}

\begin{remark}
\label{rem:cramer}
Let $R_t:=\bigl(a_\kappa-\tfrac12\sigma_\kappa^{2}\bigr)t+\sigma_\kappa W_t$ be the logarithmic return of the investment portfolio, so that $\mathcal E_t=e^{R_t}$ in~\eqref{eq:Yexpl}, and let $\varphi(s):=\log\bbe e^{sR_1}=s\bigl(a_\kappa-\tfrac12\sigma_\kappa^{2}\bigr) +\tfrac12s^{2}\sigma_\kappa^{2}$ be its Laplace exponent. Then
\begin{equation*}
    \varphi(-s)=0,\ s>0
    \quad\Longleftrightarrow\quad
    s=\frac{2a_\kappa-\sigma_\kappa^{2}}{\sigma_\kappa^{2}}=\gamma-1,
\end{equation*}
so that the tail exponent of Proposition~\ref{prop:ruin_asymptotics} is the positive root $s=\gamma-1$ of the Cram\'er-type condition $\bbe e^{-sR_1}=1$ for the logarithmic return of the portfolio, that is, the Kesten--Goldie exponent of the discount factor $\mathcal E_1^{-1}$ (below, the Cram\'er root). It is not the adjustment coefficient of the classical model, which solves the Lundberg equation for the claim surplus process, and the nondegeneracy condition~\eqref{eq:nondegeneracy} means precisely the positivity of this root. This is the form in which the exponent arises in the general theory: for the Sparre Andersen model with investments in an arbitrary L\'evy process, the tail exponent is the Cram\'er root of the Laplace exponent of the log price, and the distribution of the interclaim intervals affects only the constant~\cite{Promyslov2026}. The same root gives the exponent of the power-law tail of the corresponding perpetuity~\cite{Grincevicius,Goldie1991,BDM}.

Second, the heavy tail is already generated by a single product. If $\vartheta_1$ denotes the first interval between claims, then the discounted first claim is $\xi_1e^{-R_{\vartheta_1}}$. The exponential distribution belongs to the class $\mathcal L(\mu)$, the factor $e^{-R_{\vartheta_1}}$ is unbounded above, and the claim distribution function is continuous, so by Theorem~1.1 of~\cite{Tang2008} the product belongs to the class $\mathcal L(0)$, and every distribution of this class has a heavy tail, that is, has no finite exponential moment at any positive exponent. Theorem~1.2 of the same paper, which gives the subexponentiality of the product, is not applicable to exponential claims: it requires $F\in\mathcal S(\alpha)$ for some $\alpha>0$ (the letter $\alpha$ is taken instead of $\gamma$ of~\cite{Tang2008} so as not to confuse this exponent with the structural parameter of this paper), and membership in the class $\mathcal S(\alpha)$ requires a finite limit of the ratio $\overline{F^{\ast2}}(x)/\overline F(x)$ as $x\to\infty$. For the exponential distribution the only possible exponent is $\alpha=\mu$, but $\overline{F^{\ast2}}(x)=e^{-\mu x}(1+\mu x)$, so this ratio equals $1+\mu x$ and is unbounded. Consequently, $F\notin\mathcal S(\mu)$. In any case, the heaviness of the tail of $\Psi$ is generated by the interplay of the insurance and financial risks, and not by either of them separately; the same effect can be observed in~\cite{ChengKW2022,YangFanYuen2023}.
\end{remark}

The decay rate in~\eqref{eq:tail} is not new: it is Theorem~1 of~\cite{Frolova2002}, and in our model it follows from that result by the change of parameters~\eqref{eq:effective}. The proof given makes this self-contained: we establish the strict positivity of $K_1$, on which the leading term depends, from the equation itself in Lemma~\ref{lem:K1pos}, without any recourse to the probabilistic asymptotics.

\section{The constant in the power-law asymptotics and the limits of its applicability}
\label{sec:laplace}

Proposition~\ref{prop:ruin_asymptotics} delivers the constant $\cK$ but says nothing about its magnitude. The same is true of the entire preceding literature: the constant exists and is positive~\cite[Theorem~1]{Frolova2002}, \cite[Theorem~2.1]{KabPerg2016}, \cite[Theorem~2.1]{KPuk}, but it has not been computed, and in~\cite[Remark~3.2]{Albrecher2012} this is noted as an open question. A computable representation for it (as the limit of a solution of a pair of integral equations, the second of which is a Volterra equation of the second kind) was obtained by Promyslov~\cite{Promyslov2601,P2026} and by Antipov and Kabanov~\cite{AK2026}. In this section we find out what the constant is: the central connection coefficient of an explicit second-order equation, and we show that for integer $\gamma$ this problem can be solved in closed form. We first state precisely what was known before: the passage to the Laplace transform, the second order of the resulting equation, and the exponents $\{0,\gamma-1\}$ at the origin were established by Antipov~\cite{AntipovLaplace} for a much more general class: two-sided jump densities with rational Laplace transforms. Proposition~\ref{prop:laplace} is the one-sided case of this reduction. We write out the coefficients explicitly, since the analysis below needs them, and we use the full local picture (Proposition~\ref{prop:laplace}\,(iii) and Remarks~\ref{rem:minusmu}, \ref{rem:srank}) as the starting point for what is genuinely new in this section: the identification of the constant with a connection coefficient (Theorem~\ref{thm:connection}), the correction series (Proposition~\ref{prop:validity}), and the closed formula for integer $\gamma$ (Theorem~\ref{thm:closure}).

\subsection{The equation for the Laplace transform}

\begin{proposition}[Equation for the Laplace transform]
\label{prop:laplace}
Let $\gamma>1$ and $\wh\cH(t):=\int_0^{\infty}e^{-tu}\cH(u)\,\dd u$. Then:
\begin{enumerate}
\item[\rm(i)] the integral converges absolutely for $\operatorname{Re}t\ge0$, the function $\wh\cH$
is holomorphic in $\operatorname{Re}t>0$ and extends holomorphically to the plane with the cut
$\bbc\setminus(-\infty,0]$, which is simply connected and contains neither of the two finite singular points, and moreover $t\wh\cH(t)\to\cH(0)=1$ as $t\to+\infty$;
\item[\rm(ii)] $\wh\cH$ satisfies the equation
\begin{equation}
    t(t+\mu)\,\wh\cH''+(2-\gamma)(t+\mu)\,\wh\cH'
    +\bigl(\beta t+\mu\beta-\nu\bigr)\wh\cH=\beta ,
    \label{eq:laplace_eq}
\end{equation}
equivalently, after cancellation of the factor $t+\mu$ in the two leading coefficients,
\begin{equation}
    t\,\wh\cH''+(2-\gamma)\,\wh\cH'
    +\frac{\beta t+\mu\beta-\nu}{t+\mu}\,\wh\cH=\frac{\beta}{t+\mu} ;
    \label{eq:kernel}
\end{equation}
the quantity $\cH'(0)$ does not enter the right-hand side;
\item[\rm(iii)] the homogeneous part of~\eqref{eq:kernel} has exactly three singular points: $t=0$ and
$t=-\mu$ are regular, while $t=\infty$ is irregular and ramified, of Poincar\'e rank $\tfrac12$ in the
generalized sense of~\cite[Chapter~1]{Slavyanov2002}, which admits half-integer values at
ramified points, with subnormal solutions
\begin{equation}
    t^{(2\gamma-3)/4}\exp\bigl(\pm2i\sqrt{\beta}\sqrt t\,\bigr)
    \bigl(1+O(t^{-1/2})\bigr),\qquad t\to+\infty .
    \label{eq:subnormal}
\end{equation}
\end{enumerate}
\end{proposition}

\begin{proof}
We give the proof in Appendix~\ref{app:laplace}.
\end{proof}

\begin{remark}[The point $t=-\mu$ as a trace of the claim stream]
\label{rem:minusmu}
At the point $t=-\mu$ the exponents of~\eqref{eq:kernel} are $\{0,1\}$. Substituting $\sum_kc_k(t+\mu)^{k}$ into the homogeneous part of~\eqref{eq:laplace_eq} and collecting the constant term, we obtain the obstruction $-\nu c_0=0$. Since $\nu=2\lambda/\sigma_\kappa^{2}>0$, no solution is analytic and nonvanishing at this point, so the second local solution always contains a logarithm and the singularity is never apparent. The point $t=-\mu$ is thereby a trace of the claim stream on the transform side (its obstruction constant is proportional to $\lambda$) and it corresponds to the recessive mode $e^{-\mu u}(\mu u)^{-2}$ in~\eqref{eq:connection}. Equation~\eqref{eq:kernel} has exactly one such point. In the presence of jumps of both signs (the mixed model) there would be two, one for each direction.
\end{remark}

\begin{remark}[The accessory parameter]
\label{rem:srank}
The total $s$-rank of~\eqref{eq:kernel} equals $1+1+3/2=7/2$: strictly above the hypergeometric class, whose total $s$-rank is three, and strictly below the Heun class, where it is four. Equation~\eqref{eq:kernel} is thus a reduced confluent Heun equation. More essential is the following: the positions of the singular points and all local exponents and exponential factors ($-\mu$, the exponents $\{0,\gamma-1\}$ and $\{0,1\}$, the data~\eqref{eq:subnormal} at infinity) depend only on $\gamma$, $\beta$, and $\mu$, whereas $\nu$, that is, the claim intensity, enters none of them: of the four free parameters three are fixed by the local data, and $\nu$ remains free. A parameter of a linear equation that is determined neither by the positions of the singular points nor by the local exponents at them is customarily called accessory~\cite{Lay,Slavyanov2002}. Thus $\nu$ is the only accessory parameter of equation~\eqref{eq:kernel}. It is not quite invisible (by Remark~\ref{rem:minusmu} it is $\nu$ that decides whether the singularity at $t=-\mu$ is apparent) but it enters no exponent. Connection coefficients of equations with accessory parameters are not functions of the local exponents and are not expressible by Gauss-type formulas. That is why no closed expression for $\cK$ should be expected in general. This is an explanation, not a proof of impossibility.
\end{remark}

\begin{lemma}[Exponents at the origin]
\label{lem:exponents}
Let $\gamma>1$. The point $t=0$ is a regular singular point of equation~\eqref{eq:kernel}. Its indicial equation is $s(s+1-\gamma)=0$, so the exponents are
\begin{equation}
    s\in\{0,\ \gamma-1\},
    \label{eq:exponents}
\end{equation}
and the corresponding Frobenius series converge for $|t|<\mu$.
\end{lemma}

\begin{proof}
Near $t=0$ equation~\eqref{eq:kernel} has the form
$t\wh\cH''+(2-\gamma)\wh\cH'+O(1)\wh\cH=O(1)$, that is, after division by $t$ the coefficient of
$\wh\cH'$ has a simple pole with residue $2-\gamma$, and the coefficient of $\wh\cH$ is bounded.
Thus the Fuchs conditions are satisfied. Substituting $\wh\cH\sim t^{s}$ and collecting the terms of order $t^{s-1}$,
we obtain $s(s-1)+(2-\gamma)s=s(s+1-\gamma)=0$, which gives~\eqref{eq:exponents}. The convergence
of the Frobenius series up to the nearest singular point $t=-\mu$ is a classical statement.
\end{proof}

The exponent of the power-law tail of the ruin probability is thus an exponent at zero
of the equation satisfied by the Laplace transform of the survival density.
Antipov made this observation~\cite{AntipovLaplace}: for a much more general class
of jump densities with rational Laplace transforms he writes out the indicial equation of the
Laplace transform and finds the same exponents $0$ and $2a/\sigma^{2}-1$. The analysis carried out
there, however, is conducted under $\gamma<2$, whereas below the exponents~\eqref{eq:exponents}
serve for all $\gamma>1$, and in all examples of Section~\ref{sec:numerics} we have
$\gamma>2$. For this reason we prove Lemma~\ref{lem:exponents} here anew.

\subsection{The correction series and the domain of applicability of the asymptotics}

Before turning to the connection problem, we compute the correction series of the power-law mode. The proof of Theorem~\ref{thm:connection} needs it, and it answers a question of independent interest: starting from which capital values does~\eqref{eq:tail} describe $\Psi$?

\begin{proposition}[Correction series of the power-law mode]
\label{prop:validity}
Let $z=\mu u$. The formal power-type solution of~\eqref{eq:DCHE} at infinity is $z^{-\gamma}\sum_{k\ge0}h_kz^{-k}$ with $h_0=1$ and
\begin{equation}
    k\,h_k=\bigl[(\gamma+k-1)(k-2)+\mu\beta+\gamma-\nu\bigr]h_{k-1}
           -\mu\beta(\gamma+k-2)h_{k-2},\qquad k\ge1,
    \label{eq:hrec}
\end{equation}
where $h_{-1}:=0$. In particular
\begin{equation}
    h_1=\mu\beta-\nu=\nu\theta,
    \qquad
    h_2=\tfrac12\bigl(h_1^{2}-\gamma\nu\bigr)=\tfrac12\nu\bigl(\nu\theta^{2}-\gamma\bigr).
    \label{eq:h1h2}
\end{equation}
Moreover, with $C_\ast:=CK_1\mu^{-\gamma}$,
\begin{equation}
    g(u)=C_\ast u^{-\gamma}
         \Bigl(1+\frac{h_1}{\mu u}+\frac{h_2}{(\mu u)^{2}}+O(u^{-3})\Bigr),
    \label{eq:gcorr}
\end{equation}
\begin{equation}
    \Psi(u)=\cK\,u^{-(\gamma-1)}
    \Bigl(1+\frac{\gamma-1}{\gamma}\frac{h_1}{\mu u}
           +\frac{\gamma-1}{\gamma+1}\frac{h_2}{(\mu u)^{2}}+O(u^{-3})\Bigr).
    \label{eq:Psicorr}
\end{equation}
\end{proposition}

\begin{proof}
Substituting $w=z^{-\gamma}\sum_kh_kz^{-k}$ into~\eqref{eq:DCHE} and collecting the coefficient of $z^{-\gamma-k+1}$, we obtain~\eqref{eq:hrec}. For $k=1$ the coefficient of $z^{-\gamma}$ equals $\gamma(\gamma+1)-(\gamma+1)h_1-\gamma(\gamma+2)+\gamma h_1+(\mu\beta+\gamma-\nu) =-h_1+\mu\beta-\nu$, whence $h_1=\mu\beta-\nu$. The identity $\mu\beta-\nu=\nu\theta$ follows from~\eqref{eq:threeparams}, and $k=2$ gives $h_2$. Relation~\eqref{eq:gcorr} is the statement of the theorems on asymptotic existence of solutions, which we apply in Appendix~\ref{app:infty} to the power-law mode: the formal series~\eqref{eq:hrec} is the complete Poincar\'e series of this mode (Theorem~\ref{thm:fullseries} below establishes this together with the analogous statement for the recessive mode), the expansions can be differentiated term by term, and termwise integration of~\eqref{eq:gcorr} over $(u,\infty)$ gives~\eqref{eq:Psicorr}, which is legitimate for asymptotic power series with exponents $\gamma+k>1$.
\end{proof}

\begin{corollary}[Range of validity of the asymptotics]
\label{cor:validity}
The relative error of the leading term of~\eqref{eq:Psicorr} is $(1-\gamma^{-1})|h_1|/(\mu u)+O(u^{-2})$, so the power-law asymptotics is quantitatively accurate only for
\begin{equation}
    \mu u\gg|h_1|=\nu|\theta|,
    \qquad\text{that is}\qquad
    u\gg\frac{2\lambda|\theta|}{\mu\sigma_\kappa^{2}} ,
    \label{eq:validity}
\end{equation}
with $\Psi$ approaching its asymptotics from above for $\theta>0$ and from below for $\theta<0$.
\end{corollary}

\begin{theorem}[Complete asymptotic series at infinity]
\label{thm:fullseries}
Let $\gamma>1$ and $z=\mu u$. Then the formal solutions of~\eqref{eq:DCHE} at infinity
\begin{equation}
    z^{-\gamma}\sum_{k\ge0}h_kz^{-k}
    \qquad\text{and}\qquad
    e^{-z}z^{-2}\sum_{k\ge0}r_kz^{-k}
    \label{eq:twomodes}
\end{equation}
with the recurrences~\eqref{eq:hrec} and
\begin{equation}
    r_1=\gamma+\nu-2,\qquad
    k\,r_k=\bigl(k(\gamma-k-1)+\nu\bigr)r_{k-1}+k\,\mu\beta\,r_{k-2},
    \quad k\ge2,
    \label{eq:rrec}
\end{equation}
($r_0=1$, $r_{-1}:=0$) are complete Poincar\'e-type asymptotic expansions of the fundamental system $\{g_1,g_2\}$ from Lemma~\ref{lem:infty}: for every $N\ge1$
\begin{equation}
    \begin{aligned}
    g_1(u)&=(\mu u)^{-\gamma}\Bigl(\sum_{k=0}^{N-1}h_k(\mu u)^{-k}+O(u^{-N})\Bigr),\\
    g_2(u)&=e^{-\mu u}(\mu u)^{-2}\Bigl(\sum_{k=0}^{N-1}r_k(\mu u)^{-k}+O(u^{-N})\Bigr).
    \end{aligned}
    \label{eq:fullseries}
\end{equation}
Both series are of Gevrey class~$1$ in $z^{-1}$ and, generically, diverge: as in Lemma~\ref{lem:local0}\,(iii), a convergent series corresponds to the recessive solution of the corresponding recurrence, that is, to the vanishing of the Stokes multiplier, and the exceptional set of parameters has codimension at least one. The expansions~\eqref{eq:fullseries} can be differentiated term by term.
\end{theorem}

\begin{proof}
We give the proof in Appendix~\ref{app:fullseries}.
\end{proof}

The threshold~\eqref{eq:validity} depends not on the tail exponent but on the ratio of the absolute value of the safety loading to the effective portfolio variance, and grows as $\sigma_\kappa^{-2}$ as the strategy becomes more conservative. It is large precisely where the tail is lightest, and the two effects do not compensate each other. For the set~\eqref{eq:params} it equals $355.6$, $88.9$, and $32.0$ units of the average claim size for $\kappa=0.3$, $0.6$, and $1.0$, whereas the practically interesting range of capital for this insurer is $u\lesssim40$ (the industry-average premium-to-capital ratio $c/u\approx0.6$--$0.83$ of Remark~\ref{rem:rubles} corresponds to $u\approx13$--$18$~\cite{IRIS2024,NAIC2024}, and $u\lesssim40$ covers well-capitalized insurers as well): in such a window the log-log plot of $\Psi$ shows no straight line at all, and an intermediate slope read off it is not a power exponent (Sect.~\ref{sec:numerics}). This has no relation to the exponentially decreasing mode in~\eqref{eq:connection}, which becomes negligible much earlier. The pre-asymptotic behavior is entirely algebraic. A numerical check of~\eqref{eq:h1h2}: the quantity $\cH(u)(\mu u)^{\gamma}$ divided by the correction series $\sum_{k\ge0}h_k(\mu u)^{-k}$ of~\eqref{eq:twomodes} reaches a plateau: without corrections there is no plateau, and with terms up to $h_7$ retained, the spread over nested windows drops to $5\cdot10^{-15}$ (Section~\ref{sec:numerics}).

\begin{remark}[The threshold in dimensionless terms]
\label{rem:rubles}
Relative to the capital or to the premium, the threshold~\eqref{eq:validity} depends neither on the monetary unit nor on the scale of an individual claim. Indeed, $\mu\beta-\nu=2(c\mu-\lambda)/\sigma_\kappa^{2}$, whence
\begin{equation}
    u^{\ast}:=\frac{|\mu\beta-\nu|}{\mu}
    =\frac{2\,|c-\lambda/\mu|}{\sigma_\kappa^{2}} ,
    \label{eq:ustar_money}
\end{equation}
that is, the threshold equals twice the underwriting result per unit time, that is, premiums minus expected claims, divided by the variance of the portfolio return over the same unit of time. The unit of time cancels in the ratio, and the threshold has the dimension of capital. Neither the average claim size $1/\mu$ nor the intensity $\lambda$ enters~\eqref{eq:ustar_money} separately, and the drift $a_\kappa$ does not enter at all. Dividing by the reserve gives
\begin{equation}
    \frac{u^{\ast}}{u}=\frac{2|\theta|}{1+\theta}\cdot\frac{c/u}{\sigma_\kappa^{2}},
    \qquad\text{equivalently}\qquad
    \frac{u^{\ast}}{c}=\frac{2|\theta|}{(1+\theta)\sigma_\kappa^{2}} ,
    \label{eq:ustar_dimless}
\end{equation}
where $c/u$ is the ratio of the premium collected per unit time to policyholders' surplus. The ratio $u^{\ast}/c$ has the dimension of time: $c$ is the collection rate, so $u^{\ast}/c$ shows over what period the insurer collects a premium equal to the threshold. It is independent of the monetary unit, since under a change of the unit the numerator and the denominator change in the same way, and among the model parameters only the loading and the portfolio volatility enter it. Table~\ref{tab:thresh} gives $u^{\ast}/c$ in years, the volatility being annualized.

\begin{table}[H]
\caption{The applicability threshold $u^{\ast}/c$ of~\eqref{eq:ustar_dimless}, in years of premium, as a function of the portfolio volatility $\sigma_\kappa$ and the safety loading $\theta$.}
\label{tab:thresh}
\centering
\begin{center}
\begin{tabular}{cccc}
\toprule
$\sigma_\kappa$ & $\theta=0.03$ & $\theta=0.10$ & $\theta=0.20$ \\
\midrule
$0.05$  & $23$  & $73$  & $133$ \\
$0.075$ & $10$  & $32$  & $59$  \\
$0.10$  & $5.8$ & $18$  & $33$  \\
$0.15$  & $2.6$ & $8.1$ & $15$  \\
$0.25$  & $0.9$ & $2.9$ & $5.3$ \\
\bottomrule
\end{tabular}
\end{center}
\end{table}

The rows $0.075$, $0.15$, and $0.25$ correspond to the three strategies of Section~\ref{sec:numerics}: for $\sigma=0.25$ these are $\kappa=0.3$, $0.6$, and $1$. Both quantities involved are observable. The ratio $c/u$ is a reported indicator: in 2024 the ratio of net written premiums to policyholders' surplus in the US industry was $0.83$, and the regulatory threshold for an ``unusual'' value is $3$~\cite{IRIS2024,NAIC2024}. A caveat about units is, however, needed: the model premium $c$ does not include operating expenses, whereas the reported net premium does, so the model ratio corresponds to a value of order $0.6$. Read in the same way, the loading $\theta$ is determined by the combined ratio and amounts to about $0.03$ for the industry. As for the volatility, the fraction $\kappa$ in the model applies to the whole reserve, whereas the insurer invests assets exceeding its capital: industry loss reserves, that is, liabilities, not the reserve process $X$ of the model,
amount to $976.9$ billion dollars against capital of $1.13$ trillion, and the reported ratio
of loss reserves to capital is $86.3\%$~\cite{NAIC2024}, so that assets exceed capital almost twofold. As an order-of-magnitude estimate, the volatility of capital equals the volatility of assets multiplied by this ratio, and a conservative portfolio with volatility of $3$--$5\%$ corresponds to $\sigma_\kappa\approx0.06$--$0.10$.

For these values the threshold corresponds to a premium collected over a period from several years to several decades, whereas the insurer's capital corresponds to one or two years of premiums: the threshold exceeds the capital severalfold, and for a conservative portfolio by an order of magnitude and more. Power-law behavior is reached only for $\sigma_\kappa$ of the order corresponding to investing all capital in stocks: setting $u^{\ast}=u$ in~\eqref{eq:ustar_dimless}, we obtain $\sigma_\kappa=\sqrt{2|\theta|(c/u)/(1+\theta)}$, that is, $0.2$--$0.5$ for realistic $\theta$ and $c/u$. The line of business does not affect the ratio $u^{\ast}/u$: passing from a catastrophe line to a high-frequency personal line changes $\lambda$ and $1/\mu$ by several orders of magnitude each, but only $\theta$, $c/u$, and $\sigma_\kappa$ enter~\eqref{eq:ustar_dimless}. Finally, the absolute value in~\eqref{eq:ustar_money} is essential: a negative safety loading does occur in practice. The combined ratio of the US industry was $101.7\%$ in 2023 against $96.9\%$ in 2024~\cite{NAIC2024}.
\end{remark}

\subsection{The constant as a connection coefficient}

\begin{theorem}[Connection formula]
\label{thm:connection}
Among the solutions of~\eqref{eq:laplace_eq} on $(0,\infty)$ the function $\wh\cH$ is the unique one for which $\wh\cH(t)=O(t^{-1})$ as $t\to+\infty$. Near $t=0$ its structure is as follows.
\begin{enumerate}
\item[\rm(i)] If $\gamma\notin\bbz$, then
\begin{equation}
    \wh\cH(t)=A(t)+B\,t^{\gamma-1}\bigl(1+O(t)\bigr),\qquad t\to0+,
    \label{eq:Bdef}
\end{equation}
where $A$ is holomorphic in $\{|t|<\mu\}$. The expansion is unique, and
\begin{equation}
    C_\ast=\frac{C\,B}{\Gamma(1-\gamma)},
    \qquad
    \cK=-\frac{C\,B}{\Gamma(2-\gamma)} .
    \label{eq:KB}
\end{equation}
\item[\rm(ii)] If $\gamma=N\in\{2,3,\dots\}$, then the exponents~\eqref{eq:exponents}
differ by an integer, the second branch acquires a logarithm,
$\wh\cH(t)=A(t)+\Lambda\,t^{N-1}\log t\,(1+O(t))$ with $A$ holomorphic in $\{|t|<\mu\}$, and
$C_\ast=(-1)^{N}(N-1)!\,C\,\Lambda$.
\end{enumerate}
In both cases the constant $\cK$ from~\eqref{eq:tail} is the central connection coefficient (on the central two-point connection problem see~\cite{Lay}) of equation~\eqref{eq:kernel} between the irregular point $t=\infty$, where $\wh\cH$ is singled out by the condition $\wh\cH=O(t^{-1})$, and the regular point $t=0$, where the tail is read off from the branch with exponent $\gamma-1$.
\end{theorem}

\begin{proof}
\textit{Uniqueness at infinity.} Two solutions of~\eqref{eq:laplace_eq} differ
by a solution of the homogeneous equation, that is, by a combination of the subnormal
solutions~\eqref{eq:subnormal}. Any nonzero real combination oscillates with an algebraic
amplitude of order $t^{(2\gamma-3)/4}$, and $\gamma>1$ gives $(2\gamma-3)/4>-\tfrac14>-1$,
so no nonzero homogeneous solution is $O(t^{-1})$. Meanwhile $\wh\cH(t)\sim1/t$ by
Proposition~\ref{prop:laplace}\,(i).

\textit{Local structure at zero.} Substituting $\sum_{k\ge0}c_kt^{k}$
into~\eqref{eq:laplace_eq} and collecting the coefficient of $t^{k}$, we obtain
\begin{equation}
    \mu(k+1)(k+2-\gamma)c_{k+1}
    =\beta\delta_{k0}-\bigl[k(k+1-\gamma)+\mu\beta-\nu\bigr]c_k-\beta c_{k-1}.
    \label{eq:recimage}
\end{equation}
For $\gamma\notin\bbz$ the factor $k+2-\gamma$ never vanishes, so~\eqref{eq:recimage} gives a
one-parameter family, parametrized by $c_0$, of solutions of the inhomogeneous equation
holomorphic in $\{|t|<\mu\}$. Adjoining the Frobenius solution $t^{\gamma-1}(1+O(t))$ of the
homogeneous equation exhausts all solutions on $(0,\mu)$ by dimension count. The uniqueness of
the pair $(A,B)$ follows from the fact that $t^{\gamma-1}$ is not analytic. For $\gamma=N$ the
recurrence breaks off at $k=N-2$, and the resonant Frobenius theory gives the normal form~(ii)
with a unique $\Lambda$.

\textit{Identification of $B$.} Decompose $\cH=C_\ast C^{-1}u^{-\gamma}\Ind_{\{u>1\}}+\eta$. By
Proposition~\ref{prop:validity} we have $\eta\in L^{1}$ and $\eta(u)=O(u^{-\gamma-1})$, so that
$\int_0^{\infty}u^{m}|\eta|\dd u<\infty$ for every integer $m<\gamma$, and hence $\wh\eta\in C^{m}$
near $t=0$ for every such $m$. For the explicit part,
$\int_1^{\infty}e^{-tu}u^{-\gamma}\dd u=t^{\gamma-1}\Gamma(1-\gamma,t)$, and for non-integer
$\gamma$ one has
$\Gamma(1-\gamma,t)=\Gamma(1-\gamma)-t^{1-\gamma}\sum_{n\ge0}(-1)^n t^{n}/\bigl(n!\,(1-\gamma+n)\bigr)$,
so that $t^{\gamma-1}\Gamma(1-\gamma,t)=\Gamma(1-\gamma)t^{\gamma-1}$ plus an entire function.
Comparing with~\eqref{eq:Bdef}, we see that the difference
$\bigl(B-C_\ast\Gamma(1-\gamma)/C\bigr)t^{\gamma-1}(1+O(t))$ would have to belong to $C^{m}$ near
the origin for every integer $m<\gamma$, whereas $t^{\gamma-1}$ has an unbounded derivative of
order $\lceil\gamma\rceil-1$. Consequently, $B=C_\ast\Gamma(1-\gamma)/C$. Since
$\Gamma(2-\gamma)=(1-\gamma)\Gamma(1-\gamma)$, we have $(\gamma-1)\Gamma(1-\gamma)=-\Gamma(2-\gamma)$,
and $\cK=C_\ast/(\gamma-1)$ from~\eqref{eq:tail} gives the second relation in~\eqref{eq:KB}.
For $\gamma=N$ we use $\int_1^{\infty}e^{-tu}u^{-N}\dd u=E_N(t)$, whose expansion contains a
single nonsmooth term $(-1)^{N}t^{N-1}(\log t)/(N-1)!$. The same comparison gives
$\Lambda=(-1)^{N}C_\ast/\bigl(C\,(N-1)!\bigr)$, and multiplication by $(-1)^{N}(N-1)!$ gives
$C_\ast=(-1)^{N}(N-1)!\,C\,\Lambda$.
\end{proof}

\subsection{The connection problem in closed form for integer $\gamma$}
\label{ssec:closure}

For integer $\gamma$ the connection problem degenerates and $\cK$ becomes computable. The mechanism is the existence of a polynomial in the kernel of the operator formally adjoint to the integro-differential operator, a special case of quasi-exact solvability~\cite{Turbiner,BenderDunne}. Set $s:=\tfrac12\sigma_\kappa^{2}$ and
\begin{equation}
    \cP\pi:=\bigl(s\,x^{2}\pi\bigr)''-\bigl((a_\kappa x+c)\pi\bigr)'
                   +\lambda\sum_{k\ge1}\frac{\pi^{(k)}(x)}{\mu^{k}} ,
    \label{eq:adjoint}
\end{equation}
which is the operator formally adjoint to~\eqref{eq:generator}, acting on polynomials: for a
polynomial $\pi$ we have
$\lambda\int_0^{\infty}\pi(x+y)\mu e^{-\mu y}\dd y=\lambda\sum_{k\ge0}\pi^{(k)}(x)\mu^{-k}$,
and the term with $k=0$ cancels $-\lambda\pi$. For one-sided jumps the Fubini exchange produces
no boundary defect, so the adjointness of the integral term is exact. The cancellation
of the terms with $\lambda\pi$ is precisely the statement that $F$ is a probability measure.

\begin{proposition}[Polynomial kernel]
\label{prop:poly}
The operator $\cP$ maps polynomials of degree at most $d$ to polynomials of degree at most $d$,
and the coefficient of $x^{d}$ in $\cP(x^{d}+\dots)$ is
\begin{equation}
    (d+1)\,s\,(d+2-\gamma).
    \label{eq:leading}
\end{equation}
If $\gamma\notin\bbz$, then $\cP$ is injective on polynomials. If $\gamma=N\in\{2,3,\dots\}$,
then the kernel of $\cP$ on polynomials is one-dimensional and spanned by the monic polynomial
$\pi_N$ of degree exactly $N-2$ with coefficients rational in the model parameters. The first
three are
\begin{equation}
    \pi_2=1,\qquad
    \pi_3(x)=x-\frac{h_1}{\mu},\qquad
    \pi_4(x)=x^{2}-\frac{h_1}{\mu}x+\frac{h_1^{2}+2\nu}{2\mu^{2}} ,
    \label{eq:piN}
\end{equation}
where $h_1=\mu\beta-\nu=\nu\theta$ as in~\eqref{eq:h1h2}.
\end{proposition}

\begin{proof}
Each component of~\eqref{eq:adjoint} preserves or lowers the degree, and the coefficient of $x^{d}$ in $\cP(x^{d})$ is $s(d+2)(d+1)-a_\kappa(d+1)=(d+1)s(d+2-\gamma)$, since $\gamma=a_\kappa/s$. This is~\eqref{eq:leading}. In the monomial basis the matrix of $\cP$ is triangular with diagonal~\eqref{eq:leading}, $d=0,1,\dots$ For $\gamma\notin\bbz$ no diagonal element vanishes. For $\gamma=N$ exactly the element with $d=N-2$ vanishes, and solving the resulting triangular system gives the unique monic $\pi_N$. Formulas~\eqref{eq:piN} are a direct computation.
\end{proof}

\begin{theorem}[Closed formula for integer $\gamma$]
\label{thm:closure}
Let $\gamma=N\in\{2,3,\dots\}$, let $\pi:=\pi_N$ be as in Proposition~\ref{prop:poly}, and set
\begin{equation*}
    P_1:=\int_0^{\infty}\pi(x)e^{-\mu x}\dd x
        =\sum_{k=0}^{N-2}\frac{\pi^{(k)}(0)}{\mu^{k+1}} .
\end{equation*}
Then the constant $\cK$ from~\eqref{eq:tail} satisfies
\begin{equation}
    \tfrac12\sigma_\kappa^{2}\,(N-1)\,\cK=\lambda P_1-c\,\pi(0)\,\Psi(0),
    \qquad\text{that is}\qquad
    \cK=\frac{2\bigl[\lambda P_1-c\,\pi(0)\Psi(0)\bigr]}{\sigma_\kappa^{2}(N-1)} .
    \label{eq:closure}
\end{equation}
Equivalently, since $\Psi(0)=C\int_0^{\infty}\cH$ and $\cK=CK_1\mu^{-N}/(N-1)$, the connection
coefficient of the doubly confluent Heun equation~\eqref{eq:ODE_final} is an explicit affine
function of the normalization quadrature:
\begin{equation}
    K_1=\frac{2\mu^{N}}{\sigma_\kappa^{2}}
        \Bigl[c\,P_1+\bigl(\lambda P_1-c\,\pi_N(0)\bigr)\int_0^{\infty}\cH(v)\,\dd v\Bigr].
    \label{eq:K1closed}
\end{equation}
\end{theorem}

\begin{proof}
By Lemma~\ref{lem:idePsi} the function $\Psi$ satisfies equation \eqref{eq:IDEPsi} on $(0,\infty)$.
Fix $0<\eps<R<\infty$, multiply by $\pi$, and integrate over $[\eps,R]$. Two integrations by parts
give
\begin{align*}
    &\int_\eps^{R}\pi\Bigl(s\,x^{2}\Psi''+(a_\kappa x+c)\Psi'-\lambda\Psi\Bigr)\dd x\\
    &=\int_\eps^{R}\Psi\Bigl[\bigl(s\,x^{2}\pi\bigr)''
      -\bigl((a_\kappa x+c)\pi\bigr)'-\lambda\pi\Bigr]\dd x+\bigl[\mathcal B\bigr]_\eps^{R},
\end{align*}
where $\mathcal B(x)=s\,x^{2}\pi\Psi'-\bigl(s\,x^{2}\pi\bigr)'\Psi+(a_\kappa x+c)\pi\Psi$.
Since $\cP\pi=0$, the square bracket equals
$-\lambda\sum_{k\ge0}\pi^{(k)}\mu^{-k} =-\lambda\mu e^{\mu x}\int_x^{\infty}\pi(y)e^{-\mu y}\dd y$.
For the integral term, Fubini's theorem gives
\begin{equation*}
    \int_\eps^{R}\pi(x)\lambda\mu e^{-\mu x}\!\!\int_0^{x}\!\Psi(z)e^{\mu z}\dd z\,\dd x
    =\lambda\mu\int_0^{R}\Psi(z)e^{\mu z}
      \Bigl(\int_{z\vee\eps}^{R}\pi(y)e^{-\mu y}\dd y\Bigr)\dd z .
\end{equation*}
Adding the two contributions, the integrals over $[\eps,R]$ merge and there remains
\begin{multline*}
    \lambda\mu\Bigl(\int_0^{\eps}\Psi e^{\mu z}\dd z\Bigr)
              \Bigl(\int_\eps^{R}\pi e^{-\mu y}\dd y\Bigr)
    -\lambda\mu\Bigl(\int_\eps^{R}\Psi e^{\mu z}\dd z\Bigr)
               \Bigl(\int_R^{\infty}\pi e^{-\mu y}\dd y\Bigr)\\
    +\bigl[\mathcal B\bigr]_\eps^{R}
    =-\lambda\int_\eps^{R}\pi e^{-\mu x}\dd x .
\end{multline*}
Let $\eps\downarrow0$. The first product vanishes, since $\Psi$ is bounded, and
$\mathcal B(\eps)\to c\,\pi(0)\Psi(0)$, since $x^{2}\Psi'=-x^{2}C\cH\to0$ and
$\bigl(s\,x^{2}\pi\bigr)'\Psi\to0$ by Lemma~\ref{lem:local0}. Now let $R\to\infty$.
By~\eqref{eq:gcorr} (and the termwise differentiation legitimized by
Theorem~\ref{thm:fullseries}) we have $\Psi(R)=\cK R^{1-N}(1+O(R^{-1}))$ and
$\Psi'(R)=-(N-1)\cK R^{-N}(1+O(R^{-1}))$, whereas $\pi(R)=R^{N-2}(1+O(R^{-1}))$.
The exponentially decreasing mode contributes $O(e^{-\mu R})$ everywhere. Using the equality
$a_\kappa=sN$, we obtain
\begin{equation*}
    \mathcal B(R)=-s(N-1)\cK-sN\cK+a_\kappa\cK+O(R^{-1})
                 \longrightarrow-s\,(N-1)\,\cK ,
\end{equation*}
and the second product is $O\bigl(\Psi(R)\pi(R)\bigr)=O(R^{-1})\to0$. Collecting the limits, we
obtain $-s(N-1)\cK-c\pi(0)\Psi(0)=-\lambda P_1$, which is~\eqref{eq:closure}.
Formula~\eqref{eq:K1closed} follows by eliminating $\Psi(0)$ and $C$.
\end{proof}

\begin{corollary}[The ruin probability at zero capital]
\label{cor:psi0}
Let $\gamma\ge2$, that is, $a_\kappa\ge\sigma_\kappa^{2}$. Then
\begin{equation}
\begin{aligned}
    &\Psi(0)=\frac{1}{1+\theta}
      -\frac{a_\kappa-\sigma_\kappa^{2}}{c}\int_0^{\infty}\Psi(x)\,\dd x
    \quad(\gamma>2),\\
    &\Psi(0)=\frac{1}{1+\theta}-\frac{\sigma_\kappa^{2}\cK}{2c}
    \quad(\gamma=2),
\end{aligned}
    \label{eq:psi0}
\end{equation}
and in both cases $\Psi(0)<1/(1+\theta)$. For $\theta>0$ the right-hand side is the ruin
probability at zero capital in the model with the same $c$, $\lambda$, $\mu$ and with no
investments whatsoever, so the inequality strictly lowers it. For $\theta\le0$ the identity
remains valid but yields no comparison; in the classical model ruin is then certain, and
$1/(1+\theta)\ge1$.
\end{corollary}

\begin{proof}
For $\gamma=2$ Proposition~\ref{prop:poly} gives $\pi_2\equiv1$, whence $\pi(0)=1$ and $P_1=1/\mu$. Substituting into~\eqref{eq:closure} with $N=2$ turns it into $\tfrac12\sigma_\kappa^{2}\cK=\lambda/\mu-c\,\Psi(0)$, and $\cK>0$ by Lemma~\ref{lem:K1pos} and Corollary~\ref{cor:admissible}. For $\gamma>2$ we repeat the proof of Theorem~\ref{thm:closure} with $\pi\equiv1$: now $\cP1=2s-a_\kappa=-(a_\kappa-\sigma_\kappa^{2})\neq0$, so a constant remains under the integral, the boundary term $\mathcal B(R)=sR^{2}\Psi'(R)+\bigl((a_\kappa-2s)R+c\bigr)\Psi(R)=O(R^{2-\gamma})\to0$ by~\eqref{eq:gcorr}, and $\int_0^{\infty}\Psi$ converges since $\Psi=O(u^{1-\gamma})$. The passage to the limit gives $-(a_\kappa-\sigma_\kappa^{2})\int_0^{\infty}\Psi-c\,\Psi(0)=-\lambda/\mu$, which is the first equality. Strictness already follows from $\Psi>0$. Finally, $\lambda/(c\mu)=1/(1+\theta)$ by~\eqref{eq:threeparams}.
\end{proof}

There is no vicious circle in the proof: we prove Lemma~\ref{lem:K1pos} from the first integral~\eqref{eq:HFI}, and it does not use the quantity $\Psi(0)$.

\begin{remark}
\label{rem:psi0}
For $\sigma_\kappa=0$ relation~\eqref{eq:psi0} is a known identity for the model with a deterministic interest rate~\cite{SundtTeugels1995}. Volatility enters it as a decrease of the effective rate by $\sigma_\kappa^{2}$, namely the drift of the discounting factor $\mathcal E_t^{-1}$ from~\eqref{eq:Yexpl}. As $\gamma\downarrow2$, it follows from~\eqref{eq:tail} that $\int_0^{\infty}\Psi\sim\cK/(\gamma-2)$, and the first formula passes into the second. The condition $\gamma\ge2$ is essential: as $\gamma\to1+$ ruin becomes certain~\cite{Frolova2002,Paulsen1998}, so that $\Psi(0)\to1$, and for $\theta>0$ the inequality reverses. For the set~\eqref{eq:params} with $\kappa=1$ and $a$ adjusted, we have numerically $\Psi(0)=0.975$ at $\gamma=1.04$ and $0.914$ at $\gamma=1.60$, against the classical $0.909$, and equality is attained at $\gamma^{\ast}\approx1.77$, with $\gamma^{\ast}$ depending on $\sigma$. Note also that the ordering at $u=0$ is opposite to the asymptotic one: the power-law tail~\eqref{eq:tail} lies above the exponential tail of the classical model, so the two curves intersect at a finite capital.
\end{remark}

\begin{remark}[Numerical verification]
\label{rem:closure_num}
We have verified formula~\eqref{eq:closure} on six parameter sets for $N=2,3,4,5$: the constant $\cK$ extracted from the solution of~\eqref{eq:ODE_final} via the correction series~\eqref{eq:gcorr} in the region~\eqref{eq:validity} coincides with the right-hand side of~\eqref{eq:closure} with relative error from $2.6\cdot10^{-7}$ to $4.8\cdot10^{-6}$, which corresponds to the accuracy of the quadrature used for $\int_0^{\infty}\cH$.
\end{remark}

\begin{remark}[Scope of the closed formula]
\label{rem:closure_scope}
Theorem~\ref{thm:closure} concerns integer $\gamma$, that is, a set of zero measure in the parameter space: for non-integer $\gamma$ Proposition~\ref{prop:poly} rules out the polynomial mechanism. Moreover, \eqref{eq:closure} expresses $\cK$ in terms of $\Psi(0)$ (one bounded characteristic of the same problem) and not through an elementary formula in the model parameters. The value of the formula is that the exact solution delivers $\Psi(0)$ here directly, $\Psi(0)=C\int_0^{\infty}\cH$: for integer $\gamma$ the connection coefficient $K_1$ of the doubly confluent Heun equation, usually accessible only through the unsolved central connection problem, turns out to be an affine function of a single quadrature, computed in the normalization of the solution. A practical consequence: a Monte Carlo estimate of the single probability $\Psi(0)$ is transformed via~\eqref{eq:closure} into an estimate of the tail constant with the accuracy of a bounded functional, which is far more efficient than a direct estimation of the limit $\lim u^{\gamma-1}\Psi(u)$.
\end{remark}

\section{Numerical analysis}
\label{sec:numerics}

Formula~\eqref{eq:Phi_explicit} reduces the problem to the computation of $\cH$. We first estimate the accuracy. Unless stated otherwise, we obtain the quantities reported below with $U=4000$, a grid of $800$ uniform nodes on $[u_0,5]$ and $4000$ logarithmic ones on $[5,U]$ (we describe the extraction of the constant $K_1$ separately below and perform it with $U=20000$), with tolerance $\mathrm{rtol}=3\cdot10^{-13}$ of the implicit Radau method. We have verified convergence by varying each of these parameters separately: with $U=16000$, a fourfold refined grid, and Simpson's rule replacing the trapezoidal one, the values of $\Psi(0)$ change by at most $5\cdot10^{-8}$, the shift being practically the same for all $\kappa$ and therefore almost completely canceling in differences. For comparison: the smallest distance between $\Psi(0)$ at neighboring values of $\kappa$ in the table below is $2.7\cdot10^{-5}$, exceeding the numerical error by more than a factor of $500$. Taking the correction $h_1$ into account in the analytic tail changes $\Psi(0)$ by at most $2\cdot10^{-9}$. The position of the minimum in $\kappa$ is stable as well: under the same variations, the values $\Psi(u;\kappa)$ computed at the same point $u$ differ by at most $2.3\cdot10^{-7}$, the deviation being almost the same for all $\kappa$ and therefore canceling in differences, and $\kappa^{\ast}(u)$ found on a grid in $\kappa$ of $62$ nodes covering $[0.04,\,1]$ with refinement at small $\kappa$ does not change at $u=1$, $u=5$, $u=15$, or $u=30$. For comparison, the distance between adjacent $\kappa$-curves at these capital levels ranges from $10^{-4}$ to $3\cdot10^{-3}$, exceeding the numerical error by three or four orders of magnitude. Since $u=0$ is an irregular singular point, we cannot start from a convergent series. We use the optimally truncated Gevrey-class-$1$ series of Lemma~\ref{lem:local0}, whose terms decrease up to $n\approx\beta/u$ with a truncation error of order $e^{-\beta/u}$, so that on starting at the point $u_0$ the relative error is $O(e^{-\beta/u_0})$. We take $u_0=\min\{\beta/40,\,1\}$, which guarantees an error no worse than $e^{-40}$ and, for the set~\eqref{eq:params}, gives $u_0=1$ and an error of order $e^{-\beta}$. From $u_0$ we integrate equation~\eqref{eq:ODE_final} forward by a stiff implicit method. Forward integration is stable in both critical respects, since the branch $e^{\beta/u}u^{-\gamma}$ excluded at zero decreases as $u$ grows, and at infinity the retained mode $g_1$ is dominant. We accumulate the ruin probability backward from the analytic tail, not as $1-\wt\Phi$, which would lose all significant digits for $\Psi\ll1$. We extract the constant $K_1$ from~\eqref{eq:gcorr} with corrections up to $h_7$ inclusive: the quantity $\cH(u)(\mu u)^{\gamma}$ divided by the correction series reaches a plateau whose spread over the three nested windows $u>U/20$, $u>U/8$, and $u>U/3$ at $U=20000$ does not exceed $5\cdot10^{-15}$. One correction $h_1$ does not suffice here: for $\kappa=0.3$ it shifts $K_1$ by $0.4\%$.

Time is measured in years and capital in units of the average claim size $\bbe\xi_1=1/\mu$. Take
\begin{equation}
    \lambda=10,\quad \mu=1,\quad c=11,\quad r=0.02,\quad a=0.09,\quad \sigma=0.25 ,
    \label{eq:params}
\end{equation}
so that the safety loading is $\theta=c\mu/\lambda-1=0.10>0$. The number $\lambda=10$ is not a claim count for a mass-market insurer. The calibration fixes the size of a single claim relative to the reserve: for the industry-average premium-to-capital ratio $c/u\approx0.6$--$0.83$ the reserve amounts to some $13$ to $18$ average claims, and it is this ratio, not the number of claims, that decides whether ruin is a non-negligible event at all. The set therefore describes a portfolio of few large risks relative to capital: a specialty or large-claim book, or the layer retained net of reinsurance. Remark~\ref{rem:frequency} shows what becomes of the problem when the portfolio is refined instead. We compare three strategies: $\kappa=0.3$, $\kappa=0.6$, and $\kappa=1.0$. The effective portfolio volatility $\sigma_\kappa=\kappa\sigma$ is then $0.075$, $0.15$, and $0.25$; these are exactly the rows of the table in Remark~\ref{rem:rubles} which express the threshold in years. Table~\ref{tab:params} lists the structural parameters. All of them satisfy $\gamma>1$, as required by~\eqref{eq:nondegeneracy}.

\begin{table}[ht]
\caption{Structural parameters~\eqref{eq:structural} and solution constants for the three strategies under the set~\eqref{eq:params}. Here $C$ is the normalization constant~\eqref{eq:Cdef}, $K_1$ is the connection coefficient~\eqref{eq:connection} under the normalization $\cH(0)=1$, and the last column is the threshold $\nu|\theta|/\mu$ from~\eqref{eq:validity}. The corresponding power-law tail constants~\eqref{eq:tail} are $\cK=2.579\cdot10^{20}$, $5.240\cdot10^{4}$, and $1.819\cdot10^{1}$. The orders of magnitude of $K_1$ reflect the normalization $\cH(0)=1$ and the size of $\gamma$.}
\label{tab:params}
\centering
\setlength{\tabcolsep}{4pt}
\begin{tabular}{lcccccccc}
\toprule
Strategy & $\kappa$ & $\gamma$ & $\mu\beta$ & $\nu$ & $C$ & $\Psi(0)$ & $K_1$ & $\nu|\theta|/\mu$ \\
\midrule
conservative & $0.3$ & $14.578$ & $3911.1$ & $3555.6$ & $0.10260$ & $0.88714$ & $3.4134\cdot10^{22}$ & $355.6$ \\
moderate      & $0.6$ & $\phantom{0}5.511$ & $\phantom{0}977.8$ & $\phantom{0}888.9$ & $0.10521$ & $0.88427$ & $2.2468\cdot10^{6}$ & $\phantom{0}88.9$ \\
aggressive    & $1.0$ & $\phantom{0}2.880$ & $\phantom{0}352.0$ & $\phantom{0}320.0$ & $0.10287$ & $0.88684$ & $3.3250\cdot10^{2}$ & $\phantom{0}32.0$ \\
\bottomrule
\end{tabular}
\end{table}

The first verification of the formula is available at $u=0$. In the classical Cram\'er--Lundberg model without investments $\Psi_{\mathrm{CL}}(0)=\lambda\bbe\xi_1/c=1/(1+\theta)$, and the boundary relation~\eqref{eq:boundary} has the same form as in the classical model, where together with $\Phi(\infty)=1$ it delivers precisely this value. Table~\ref{tab:params} shows that investments lower $\Psi(0)$ below the classical level and that $\Psi(0)$ is not monotone in $\kappa$.

\begin{table}[ht]
\caption{The local decay exponent $-\dd\log\Psi(u)/\dd\log u$ computed from the exact
solution, against its limit $\gamma-1$. The exponent reaches its limiting value only for
$u$ well above the threshold $\nu|\theta|/\mu$ from~\eqref{eq:validity}, given
in the second column. For the conservative strategy the practically significant range of capital
$u\lesssim40$ lies entirely within the pre-asymptotic region: at $u=20$ the local
exponent is $3.03$ against the limit $13.58$.}
\label{tab:slopes}
\centering
\setlength{\tabcolsep}{4.5pt}
\begin{tabular}{ccccccccc}
\toprule
$\kappa$ & $\nu|\theta|/\mu$ & $\gamma-1$
 & $u=20$ & $u=50$ & $u=100$ & $u=300$ & $u=1000$ & $u=3000$ \\
\midrule
$0.3$ & $355.6$ & $13.578$ & $3.031$ & $7.915$ & $12.233$ & $14.133$ & $13.862$ & $13.682$ \\
$0.6$ & $\phantom{0}88.9$ & $\phantom{0}4.511$ & $2.637$ & $4.471$ & $\phantom{0}4.856$ & $\phantom{0}4.715$ & $\phantom{0}4.580$ & $\phantom{0}4.530$ \\
$1.0$ & $\phantom{0}32.0$ & $\phantom{0}1.880$ & $1.767$ & $2.095$ & $\phantom{0}2.043$ & $\phantom{0}1.945$ & $\phantom{0}1.900$ & $\phantom{0}1.884$ \\
\bottomrule
\end{tabular}
\end{table}

Table~\ref{tab:slopes} illustrates Proposition~\ref{prop:validity}: the local exponent $-\dd\log\Psi/\dd\log u$ approaches $\gamma-1$ only for $u$ going far beyond the threshold~\eqref{eq:validity}. For the conservative strategy the threshold is of order several hundred units of the average claim size, and the asymptotic slope is therefore not attained in any practically interesting range of capital: at $u=20$ the local exponent is $3.03$ against the limit $13.58$. A log-log plot over a realistic window consequently shows no rectilinear segment, and an intermediate slope read off such a plot is not a power exponent. The exact formula~\eqref{eq:Phi_explicit}, by contrast, is valid everywhere; this is exactly the region where we need it.

We simulate the reserve claim by claim: the intervals between claims are exponential, between two claims we integrate the linear equation of Lemma~\ref{lem:origin} by formula~\eqref{eq:Yexpl} with eight substeps per interval (the multiplicative factor $\mathcal E$ is computed exactly, and the premium integral by the trapezoidal rule), and at each arrival a claim $\mathrm{Exp}(\mu)$ is subtracted. We register ruin when the reserve becomes negative. Each row of Table~\ref{tab:mc} uses $2\cdot10^{5}$ trajectories and a horizon of $400$ years. The largest deviation is $2.13$ standard errors, the others are below two.

\begin{table}[ht]
\caption{The exact formula~\eqref{eq:Phi_explicit} against the Monte Carlo method,
$2\cdot10^{5}$ trajectories per row. ``s.e.''\ is the standard error of the estimate,
$z=(\Psi_{\mathrm{MC}}-\Psi)/\text{s.e.}$}
\label{tab:mc}
\centering
\begin{tabular}{ccccc}
\toprule
$\kappa$ & $u$ & $\Psi(u)$ exact & $\Psi(u)$ Monte Carlo & $z$ \\
\midrule
$0.3$ & $\phantom{0}1$ & $0.78929$ & $0.78809\pm0.00091$ & $-1.32$ \\
$0.3$ & $\phantom{0}3$ & $0.62038$ & $0.61932\pm0.00109$ & $-0.98$ \\
$0.3$ & $\phantom{0}6$ & $0.42511$ & $0.42510\pm0.00111$ & $-0.01$ \\
$0.3$ & $12$ & $0.18910$ & $0.18929\pm0.00088$ & $+0.21$ \\
$0.3$ & $25$ & $0.02720$ & $0.02790\pm0.00037$ & $+1.88$ \\
\midrule
$0.6$ & $\phantom{0}1$ & $0.78405$ & $0.78491\pm0.00092$ & $+0.94$ \\
$0.6$ & $\phantom{0}3$ & $0.61200$ & $0.61278\pm0.00109$ & $+0.71$ \\
$0.6$ & $\phantom{0}6$ & $0.41593$ & $0.41638\pm0.00110$ & $+0.41$ \\
$0.6$ & $12$ & $0.18627$ & $0.18784\pm0.00087$ & $+1.81$ \\
$0.6$ & $25$ & $0.03338$ & $0.03350\pm0.00040$ & $+0.30$ \\
\midrule
$1.0$ & $\phantom{0}1$ & $0.78900$ & $0.78939\pm0.00091$ & $+0.43$ \\
$1.0$ & $\phantom{0}3$ & $0.62216$ & $0.62194\pm0.00108$ & $-0.20$ \\
$1.0$ & $\phantom{0}6$ & $0.43477$ & $0.43713\pm0.00111$ & $+2.13$ \\
$1.0$ & $12$ & $0.21903$ & $0.21923\pm0.00093$ & $+0.21$ \\
$1.0$ & $25$ & $0.06602$ & $0.06699\pm0.00056$ & $+1.73$ \\
\bottomrule
\end{tabular}
\end{table}

For $\Psi\lesssim10^{-4}$ the direct Monte Carlo method ceases to be informative. In this region we compare the exact formula with the asymptotics~\eqref{eq:tail}, which is the subject of Table~\ref{tab:slopes}.

The exact solution makes it possible to compare strategies at finite capital, and this is where it is most valuable. Set
\begin{equation}
    \kappa^{\ast}(u):=\argmin_{\kappa\in(0,1]}\Psi(u;\kappa).
    \label{eq:kappastar}
\end{equation}
An asymptotic analysis alone would recommend taking $\kappa$ as small as possible, since $\gamma(\kappa)=2\bigl(r+\kappa(a-r)\bigr)/(\kappa^{2}\sigma^{2})$ is strictly decreasing in $\kappa$, and a larger $\gamma$ means a lighter tail. At finite capital this is wrong: a moderate share of the risky asset raises the drift and helps the reserve escape the neighborhood of the origin, where ruin is caused by ordinary claims rather than adverse market moves. Figure~\ref{fig:kappastar} shows that the minimum over $\kappa$ is attained at an interior point of the interval $(0,1]$ and shifts to the left as the capital grows: for the parameter set~\eqref{eq:params} the value of $\kappa^{\ast}$ decreases from $0.65$ at small capitals to $0.60$ at $u=5$, $0.53$ at $u\approx9$, $0.43$ at $u\approx15$, $0.30$ at $u\approx27$, and $0.18$ at $u\approx47$, tending further to zero in agreement with the asymptotic ordering.

The preceding remark admits a sharper form, and the sharpened form matters more than the remark itself. The instruction ``take $\kappa$ as small as possible'' uses merely the exponent $\gamma(\kappa)$. If instead we use the full leading term~\eqref{eq:tail} together with the constant, that is, minimizes over $\kappa$ the expression $\cK(\kappa)u^{-(\gamma(\kappa)-1)}$, a finite recommendation results, since $\cK(\kappa)$, as computation shows, grows without bound as $\kappa\to0$. Denote it by $\kappa_{\mathrm{as}}(u)$. This is the asymptotically optimal fraction in the sense of~\cite{HultLindskog2011}. This criterion requires knowledge of $\cK$ as a function of $\kappa$. Here the same exact solution (via Lemma~\ref{lem:infty} and the correction series of Proposition~\ref{prop:validity}) delivers it, so that the ``asymptotic'' recommendation too is computed by the tools of this paper. Table~\ref{tab:asym} gives the comparison with the exact minimum.

\begin{table}[H]
\caption{The exact optimal fraction $\kappa^{\ast}(u)$ against the asymptotically optimal $\kappa_{\mathrm{as}}(u)$, and the price of following the latter.}
\label{tab:asym}
\centering
\begin{center}
\begin{tabular}{cccccc}
\toprule
$u$ & $\Psi^{\ast}(u)$ & $\kappa^{\ast}(u)$ & $\kappa_{\mathrm{as}}(u)$
    & $\Psi\bigl(u;\kappa_{\mathrm{as}}\bigr)$ & $\Psi/\Psi^{\ast}$ \\
\midrule
$35$  & $5.4\cdot10^{-3}$  & $0.24$ & $0.54$ & $8.7\cdot10^{-3}$  & $1.6$ \\
$60$  & $5.88\cdot10^{-5}$ & $0.14$ & $0.37$ & $1.73\cdot10^{-4}$ & $2.9$ \\
$80$  & $9.30\cdot10^{-7}$ & $0.10$ & $0.30$ & $5.78\cdot10^{-6}$ & $6.2$ \\
$100$ & $9.35\cdot10^{-9}$ & $0.08$ & $0.25$ & $1.38\cdot10^{-7}$ & $14.7$ \\
\bottomrule
\end{tabular}
\end{center}
\end{table}

Two effects are at work here and should not be confused. The ratio of probabilities is large only where the probabilities themselves are negligible: at $u=100$ we are talking about $10^{-9}$ versus $10^{-7}$, and there is no practical content in this comparison. The discrepancy of the recommendations themselves, however, is significant precisely in the working range: at $u=35$, where $\Psi\approx5\cdot10^{-3}$ (the order of magnitude targeted by solvency requirements), the exact criterion prescribes keeping $24\%$ of the reserve in the risky asset, the asymptotic one $54\%$, that is, more than twice as much. At $u=60$ it is $14\%$ versus $37\%$. The asymptotics errs not so much in the estimate of the ruin probability as in the choice of strategy, and errs by a factor of several exactly where the choice is made.

The reason is visible from Corollary~\ref{cor:validity}. The optimal fractions are small, while the applicability threshold~\eqref{eq:validity} grows as $\sigma_\kappa^{-2}$, that is, as $\kappa^{-2}$: for $\kappa^{\ast}(80)=0.10$ it amounts to about $3.2\cdot10^{3}$ average claims, whereas the capital itself is $80$. The asymptotic recommendation itself lies outside the applicability region: for $\kappa_{\mathrm{as}}(80)=0.30$ the threshold is $356$, also far greater than $80$. The minimum is sought in the region where the leading term~\eqref{eq:tail} does not yet describe the solution, and no refinement of the constant corrects this. What has been said concerns the criterion, not the class of strategies: the constant fraction is not optimal in itself either: in the problem with feedback the optimal investment sum depends on the reserve and only tends to a constant~\cite{HippPlum2000,HippSchmidli2004}, while the asymptotically optimal strategy is the one with a constant sum~\cite{GGS2003}, but here we compare two ways of choosing $\kappa$ within the same class.

\begin{remark}[Robustness of the discrepancy with respect to the parameters]
\label{rem:robust}
The discrepancy is not a peculiarity of the set~\eqref{eq:params}. In the table below we vary the loading,
the volatility, the risk premium, and the claim intensity one at a time and recompute $\kappa^{\ast}$ and $\kappa_{\mathrm{as}}$ in the same way. See Table~\ref{tab:robust}.

\begin{table}[H]
\caption{Robustness of the discrepancy: the safety loading, the volatility, the risk premium, and the claim intensity are varied one at a time. The minimization here uses a grid in $\kappa$ of step $0.02$, which accounts for the entry $0.38$ in the first row against $0.37$ in Table~\ref{tab:asym}.}
\label{tab:robust}
\centering
\begin{center}
\begin{tabular}{lccccccc}
\toprule
& & \multicolumn{3}{c}{$u=35$} & \multicolumn{3}{c}{$u=60$} \\
\cmidrule(lr){3-5}\cmidrule(lr){6-8}
varied parameter & $\theta$ & $\kappa^{\ast}$ & $\kappa_{\mathrm{as}}$ & $\kappa_{\mathrm{as}}/\kappa^{\ast}$ & $\kappa^{\ast}$ & $\kappa_{\mathrm{as}}$ & $\kappa_{\mathrm{as}}/\kappa^{\ast}$ \\
\midrule
--- (set~\eqref{eq:params})      & $0.10$ & $0.24$ & $0.54$ & $2.3$ & $0.14$ & $0.38$ & $2.7$ \\
$c=10.3$                         & $0.03$ & $0.32$ & $0.72$ & $2.3$ & $0.18$ & $0.48$ & $2.7$ \\
$c=12.0$                         & $0.20$ & $0.18$ & $0.40$ & $2.2$ & $0.10$ & $0.28$ & $2.8$ \\
$\sigma=0.15$                    & $0.10$ & $0.64$ & $1.00$ & $1.6$ & $0.38$ & $0.74$ & $1.9$ \\
$\sigma=0.35$                    & $0.10$ & $0.13$ & $0.34$ & $2.6$ & $0.07$ & $0.24$ & $3.4$ \\
$a=0.06$                         & $0.10$ & $0.14$ & $0.45$ & $3.2$ & $0.08$ & $0.32$ & $4.0$ \\
$\lambda=40$, $c=44$               & $0.10$ & $0.30$ & $0.56$ & $1.9$ & $0.20$ & $0.42$ & $2.1$ \\
\bottomrule
\end{tabular}
\end{center}
\end{table}

The ratio $\kappa_{\mathrm{as}}/\kappa^{\ast}$ lies between $1.6$ and $3.2$ at $u=35$ and between $1.9$ and $4.0$ at $u=60$. The fractions themselves vary by a factor of five. The overstatement is
thus not tied to the calibration. The row $\sigma=0.15$ stands out qualitatively: there
$\kappa_{\mathrm{as}}=1$ at $u=35$, that is, the asymptotic criterion hits the boundary of admissible
strategies and prescribes keeping the whole reserve in the risky asset, whereas the exact minimum
is interior. Note also the direction of the dependence: the discrepancy grows as the risk premium
$a-r$ decreases and as the volatility increases; that is, the more conservative the
optimal strategy, and accordingly the farther the threshold~\eqref{eq:validity}, the
stronger the discrepancy.
\end{remark}

The shape of the minimum is no less informative. The shaded regions in Fig.~\ref{fig:kappastar} correspond to the fractions for which the ruin probability exceeds the optimal one by at most a factor of $1.05$ and of $2$, respectively. For $u\lesssim3$ the whole range considered, $[0.04,\,1]$, fits within the five-percent tolerance: at small capital the choice of strategy is almost immaterial. At $u\approx15$ the five-percent tolerance shrinks to the interval $[0.24,\,0.64]$, and at $u\approx27$ to $[0.20,\,0.40]$. Thus, as the capital grows, not only does the optimal risky fraction decrease, but the price of deviating from it increases. Neither follows from the asymptotics~\eqref{eq:tail}, since it prescribes only taking $\kappa$ as small as possible.

The same nonmonotonicity is already visible at zero capital, where the classical formula provides a check. We collect the values of $\Psi(0)$ as a function of $\kappa$ in Table~\ref{tab:psi0}.
\begin{table}[H]
\caption{The ruin probability at zero capital as a function of the risky fraction, against the classical value $1/(1+\theta)=0.90909$.}
\label{tab:psi0}
\centering
\begin{center}
\setlength{\tabcolsep}{2.8pt}
\begin{tabular}{lccccccccc}
\toprule
$\kappa$ & $0.1$ & $0.2$ & $0.3$ & $0.4$ & $0.5$ & $0.6$ & $0.7$ & $0.85$ & $1.0$ \\
\midrule
$\Psi(0)$ & $0.8918$ & $0.8892$ & $0.8871$ & $0.8857$ & $0.8847$ & $0.8843$
 & $0.8842$ & $0.8850$ & $0.8868$ \\
\bottomrule
\end{tabular}
\end{center}
\end{table}
so that $\Psi(0)$ decreases up to $\kappa\approx0.66$ and increases after, while remaining below the classical Cram\'er--Lundberg value $\Psi_{\mathrm{CL}}(0)=1/(1+\theta)=0.90909$ for every $\kappa$. Both features are what we would expect: moderate exposure to risk helps, excessive exposure hurts. That $\Psi(0)$ stays below the classical level is no accident: Corollary~\ref{cor:psi0} proves it for $\gamma\ge2$. For $\gamma$ near one the inequality reverses (Remark~\ref{rem:psi0}). Neither feature is accessible without the exact solution.

We establish the existence of an interior minimum at $u=0$ analytically, from the single criterion of two strict inequalities, and verify the inequalities themselves on the exact solution with controlled error.

\begin{proposition}[Interior minimum at zero capital]
\label{prop:kappastar0}
Let~\eqref{eq:nondegeneracy} hold for all $\kappa\in(0,1]$, that is, $\gamma(\kappa)>1$. Then
the function $\kappa\mapsto\Psi(0;\kappa)$ is continuous on $(0,1]$. If, in addition, for some $0<\kappa_0<\kappa_1\le1$ we have
\begin{equation*}
    \Psi(0;\kappa_0)>\Psi(0;\kappa_1)
    \qquad\text{and}\qquad
    \Psi(0;1)>\min_{\kappa\in[\kappa_0,\kappa_1]}\Psi(0;\kappa),
\end{equation*}
then the minimum of $\Psi(0;\cdot)$ on the interval $[\kappa_0,1]$ is attained at a point of the
open interval $(\kappa_0,1)$.
\end{proposition}

\begin{proof}
\textit{Continuity.} Fix $\kappa_\ast\in(0,1]$ and a compact neighborhood $K\subset(0,1]$.
The parameters $\gamma,\beta,\nu$ are continuous in $\kappa$, and the expansions of
Lemma~\ref{lem:local0} and Lemma~\ref{lem:infty} are uniform in $\kappa\in K$, and this follows
from the construction of asymptotic solutions at an irregular singular point, all steps of
which (characteristic roots, coefficients of formal series, remainder bounds) depend
continuously on the coefficients of the equation and are therefore uniform on compact sets in the
parameter space~\cite[\S\S12--15]{Wasow}, \cite[Chapter~VII]{HsiehSibuya}. Choose $\bar u>0$
and $U\ge\bar u$ with $\mu U\ge1$. Since as $u\to0+$ the optimally truncated series of
Lemma~\ref{lem:local0}\,(ii) approximates $\cH_\kappa$ with a remainder bound uniform in the
parameter and uniform over $\kappa\in K$ (the remainder bound depends only on $\beta$, bounded
on the compact), the values $\cH_\kappa(\bar u)$ and $\cH'_\kappa(\bar u)$ are continuous in
$\kappa$, and then the regular Cauchy problem on $[\bar u,U]$ with initial data depending
continuously on the parameter gives continuity of $\cH_\kappa$ on $[\bar u,U]$ uniform in
$\kappa$. At infinity, for all $u\ge U$ and $\kappa\in K$
\begin{equation*}
    \cH_\kappa(u)=K_1(\kappa)(\mu u)^{-\gamma(\kappa)}
    \Bigl(1+O\bigl((\mu u)^{-1}\bigr)\Bigr),
\end{equation*}
where the $O$-term is uniform in $\kappa$ and has absolute value at most $\tfrac12$.
Therefore the quantity $K_1(\kappa)$ is bounded on $K$ by
\begin{equation*}
    K_1(\kappa)\le 2(\mu U)^{\gamma(\kappa)}\cH_\kappa(U)\le M<\infty,
\end{equation*}
since $\gamma(\kappa)$ is bounded on the compact and $\kappa\mapsto\cH_\kappa(U)$ is continuous.
Then for $u\ge U$, using $\mu u\ge1$,
\begin{equation*}
    \cH_\kappa(u)\le 2M(\mu u)^{-\gamma_-},
    \qquad \gamma_-:=\min_{\kappa\in K}\gamma(\kappa)>1,
\end{equation*}
and the right-hand side is integrable on $(U,\infty)$. For $\kappa_n\to\kappa_\ast$ from $K$ we
have $\cH_{\kappa_n}\to\cH_{\kappa_\ast}$ pointwise (on $[0,U]$ by construction, on
$(U,\infty)$ by the uniform asymptotics), so the dominated convergence theorem gives the
continuity of $I(\kappa):=\int_0^\infty\cH_\kappa$, and with it the continuity of
$C(\kappa)=(c/\lambda+I(\kappa))^{-1}$ and $\Psi(0;\kappa)=1-C(\kappa)c/\lambda$ on $(0,1]$.

\textit{Interiorness.} By continuity the minimum $m$ on the compact $[\kappa_0,1]$ is attained
at some $\kappa^{\ast}$. From $\Psi(0;\kappa_1)<\Psi(0;\kappa_0)$ it follows that
$\kappa^{\ast}\neq\kappa_0$. Further, $\min_{[\kappa_0,\kappa_1]}\Psi(0;\cdot)\ge m$ by the
nesting of the intervals, hence $\Psi(0;1)>m$ and $\kappa^{\ast}\neq1$. Consequently
$\kappa^{\ast}\in(\kappa_0,1)$, an interior point of $(0,1)$.
\end{proof}

For the set~\eqref{eq:params} the conditions of Proposition~\ref{prop:kappastar0} are satisfied: $\Psi(0;0.1)=0.89180$, $\Psi(0;0.7)=0.88424$, and $\Psi(0;1)=0.88684$. The numerical profile is monotone on $[0.04,\,0.66]$, so the minimum over the whole computed range is interior as well.

\begin{remark}
As $\kappa\to0+$ the reserve converges not to the Cram\'er--Lundberg process but to the process with deterministic rate $r$: the limiting model is $\dd X_t=(c+rX_t)\dd t-\dd Z_t$. For it, the first integral with $\sigma_\kappa=0$ gives $g(u)=A(c+ru)^{\lambda/r-1}e^{-\mu u}$, and together with~\eqref{eq:boundary} and the normalization this determines $\Psi_r(0)$ explicitly. For the set~\eqref{eq:params} we obtain $\Psi_r(0)=0.89525$. It is to this value, and not to the classical $1/(1+\theta)=0.90909$, that $\Psi(0;\kappa)$ converges: the computation gives $0.8891$, $0.8904$, $0.8918$, and $0.8934$ for $\kappa=0.20$, $0.15$, $0.10$, and $0.05$. The convergence is ordinary. The difference between the limit and the classical value is the contribution of the risk-free rate, not an effect of volatility.
\end{remark}

\begin{remark}[Claim frequency and the regime of the calibration]
\label{rem:frequency}
Taken literally, the set~\eqref{eq:params} gives the insurer ten claims a year, whereas a mass-market portfolio has thousands. What governs the problem, however, is not the number of claims but their size relative to the reserve, and the two can be separated. Refine the portfolio at a fixed balance sheet: replace $\lambda$ by $N\lambda$ and the mean claim $1/\mu$ by $1/(N\mu)$, leaving $c$, $r$, $a$, $\sigma$ and the reserve in money units unchanged. The expected claim amount per unit time $\lambda/\mu$, the loading $\theta$ and the ratio $c/u$ are then unchanged, and so, by~\eqref{eq:ustar_dimless}, is the threshold ratio $u^{\ast}/u$, into which neither $\lambda$ nor $\mu$ enters. Only the granularity changes. Table~\ref{tab:frequency} gives the ruin probability at the reserve $u=35$ under such refinement.

\begin{table}[H]
\caption{Refinement of the portfolio at a fixed balance sheet: $\lambda\to N\lambda$, $1/\mu\to1/(N\mu)$, with $c=11$, $\theta=0.10$, $\sigma=0.25$ and the reserve $u=35$ in money units held fixed. The column $\mu u$ is the same reserve expressed in average claims.}
\label{tab:frequency}
\centering
\begin{center}
\begin{tabular}{rrrrccc}
\toprule
& & & & \multicolumn{3}{c}{$\Psi(35)$} \\
\cmidrule(lr){5-7}
$N$ & $\lambda$ & $1/\mu$ & $\mu u$ & $\kappa=0.3$ & $\kappa=0.6$ & $\kappa=1.0$ \\
\midrule
$1$  & $10$  & $1$    & $35$  & $5.57\cdot10^{-3}$  & $1.02\cdot10^{-2}$  & $3.37\cdot10^{-2}$ \\
$2$  & $20$  & $0.5$  & $70$  & $1.21\cdot10^{-4}$  & $8.12\cdot10^{-4}$  & $8.16\cdot10^{-3}$ \\
$5$  & $50$  & $0.2$  & $175$ & $1.51\cdot10^{-8}$  & $5.84\cdot10^{-6}$  & $5.82\cdot10^{-4}$ \\
$10$ & $100$ & $0.1$  & $350$ & $3.15\cdot10^{-13}$ & $3.01\cdot10^{-8}$  & $3.46\cdot10^{-5}$ \\
$20$ & $200$ & $0.05$ & $700$ & $9.01\cdot10^{-20}$ & $2.56\cdot10^{-11}$ & $7.18\cdot10^{-7}$ \\
\bottomrule
\end{tabular}
\end{center}
\end{table}

The Poisson component diversifies, while the investment component does not. For the conservative strategy a twentyfold refinement lowers the ruin probability by seventeen orders of magnitude, whereas for $\kappa=1$, where the tail exponent $\gamma-1=1.88$ is the smallest of the three and ruin is driven by the risky asset, it lowers it by five. Ruin is therefore an event of solvency relevance in this model only when a single claim is a material fraction of the reserve. The calibration~\eqref{eq:params}, in which the reserve amounts to some tens of average claims, is of that kind: it describes a specialty or large-claim portfolio, or the layer retained net of reinsurance, rather than a mass-market book of small claims. The conclusions of this section are statements about the first regime.
\end{remark}

\begin{figure}[ht]
    \centering
    \includegraphics[width=0.78\textwidth]{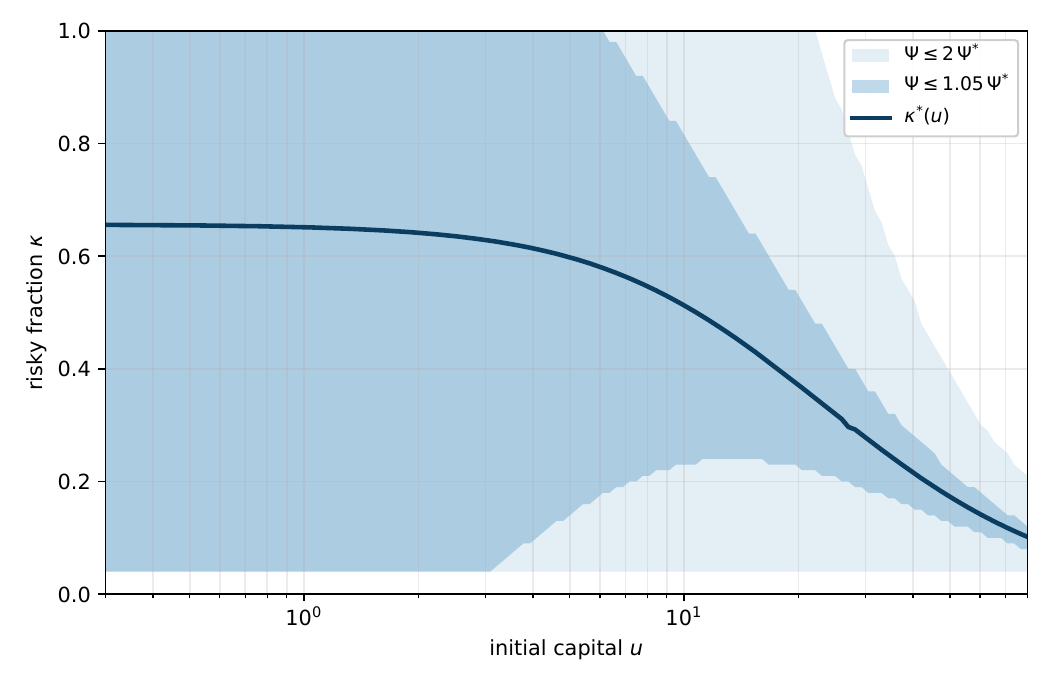}
    \caption{The optimal risky fraction $\kappa^{\ast}(u)$ from~\eqref{eq:kappastar} under the parameters~\eqref{eq:params} (solid line). The regions in which the ruin probability exceeds the optimal one by at most a factor of $1.05$ (dark) and by at most a factor of two (light) are shaded. The minimization was performed on the grid $\kappa\in[0.04,\,1]$ of $62$ nodes, so at small $u$ the lower boundaries of the bands correspond to the edge of the grid rather than to the boundary of the region.}
    \label{fig:kappastar}
\end{figure}

This effect concerns the region of finite capitals and is therefore inaccessible to asymptotic theory.

\section{Discussion}
\label{sec:conclusion}

The reduction of the IDE of the Cram\'er--Lundberg model with proportional investments and exponential claims to a doubly confluent Heun equation makes the survival probability available in closed form at any initial capital. The properties that turn the constructed function into a probability follow from the single equation through the first integral of Lemma~\ref{lem:firstintegral}. The same first integral gives the strict positivity of the connection coefficient of the power-law mode. We obtain all these properties before and independently of the identification of the constructed function with the survival probability, and therefore the verification theorem contains no logical circle.

For non-integer $\gamma$ we obtain no closed expression for the constant: Proposition~\ref{prop:poly} rules out the polynomial mechanism there, and the count of accessory parameters in Remark~\ref{rem:srank} explains why Gauss-type formulas should not be expected, but this is an argument, not a proof of impossibility, and we do not have the latter. Even for integer $\gamma$ the closed formula~\eqref{eq:closure} expresses $\cK$ in terms of the quadrature $\int_0^{\infty}\cH$ rather than through an elementary formula in the model parameters.

We single out three open problems.

The first is the degenerate case $c=0$, in which the origin becomes a regular singular point, the local analysis of Lemma~\ref{lem:local0} is inapplicable, and the indicial equation gives, in general, negative exponents, so that boundedness at zero no longer singles out a solution. For the model with stochastic premiums this case was analyzed in~\cite[\S\S3.2,\,4]{BKK2012}.

The second is jumps of both signs, that is, a model in which random income arrives along with the claims. The annihilating operator then has to be applied twice, the order of the reduced equation rises to three, and the equation for the Laplace transform acquires three finite regular singular points instead of two. The question whether the closed formula of Theorem~\ref{thm:closure} survives in this situation, and in what form, is not addressed here.

The third is claim distributions with a rational Laplace transform of general form, for which, by Remark~\ref{rem:general_F}, the reduction is also possible but leads to equations of higher order. Whether they still belong to the Heun class, and what happens there to the connection problem, remains open.

\section*{Competing interests}
The authors declare no competing interests.


\appendix\normalsize
\section{Proofs of a technical nature}
\label{app:proofs}

The appendix collects the proofs that consist in applying the classical theorems on asymptotic existence of solutions at irregular singular points~\cite{Wasow,Coddington,HsiehSibuya} and in direct computation of coefficients: these are Lemmas~\ref{lem:local0} and~\ref{lem:infty}, Proposition~\ref{prop:laplace}, and Theorem~\ref{thm:fullseries}. None of them contains ideas absent from the statements. The proofs of the results constituting the contribution of the paper, namely the first integral (Lemma~\ref{lem:firstintegral}), positivity (Lemmas~\ref{lem:positivity} and~\ref{lem:K1pos}), equivalence (Proposition~\ref{prop:equivalence}), verification (Theorem~\ref{thm:verification}), the connection formula (Theorem~\ref{thm:connection}), and the closed formula (Theorem~\ref{thm:closure}), appear in the main text.

\subsection{Local analysis at the irregular origin (proof of Lemma~\ref{lem:local0})}
\label{app:local0}

The argument relies on the theorems on asymptotic existence of solutions at irregular singular points~\cite[\S\S12--15]{Wasow}, \cite[Chapter~V]{Coddington}, and~\cite[Theorem~VII-5-3]{HsiehSibuya}.

\textit{Dominant branch.} Set $g=e^{h}$ with $h'(u)=Ku^{-2}+\rho u^{-1}+O(1)$ as
$u\to0+$. Then $g'/g=h'$ and $g''/g=h''+(h')^{2}$, and after dividing~\eqref{eq:ODE_final} by $g$ we obtain
\begin{equation}
    u^{2}\bigl(h''+(h')^{2}\bigr)
    +\bigl[\mu u^{2}+(\gamma+2)u+\beta\bigr]h'
    +\mu\gamma u+(\mu\beta+\gamma-\nu)=0 .
    \label{eq:riccati}
\end{equation}
The terms of order $u^{-2}$ in~\eqref{eq:riccati} are $K^{2}$, from $u^{2}(h')^{2}$, and
$\beta K$, from $\beta h'$, whence
\begin{equation*}
    K^{2}+\beta K=K(K+\beta)=0,
    \qquad K\in\{0,\,-\beta\} .
\end{equation*}
Both roots are distinct, since $\beta>0$, so the leading matrix of the system obtained
from~\eqref{eq:ODE_final} by the substitution $t=1/u$ splits into two one-dimensional blocks, and
the cited theorems apply to each. The root $K=-\beta$ gives $h=\beta/u+\rho\log u+O(u)$, that is,
the solution $e^{\beta/u}u^{\rho}(1+O(u))$. The exponent $\rho$ is determined by the next order:
for $K=-\beta$ we have $u^{2}\bigl(h''+(h')^{2}\bigr)=\beta^{2}u^{-2}+(2\beta-2\beta\rho)u^{-1}+O(1)$
and $\bigl[\mu u^{2}+(\gamma+2)u+\beta\bigr]h'=-\beta^{2}u^{-2} +\bigl(\beta\rho-(\gamma+2)\beta\bigr)u^{-1}+O(1)$,
so the terms of order $u^{-2}$ cancel and the terms of order $u^{-1}$ give
\begin{equation*}
    \beta\bigl[2-2\rho-(\gamma+2)+\rho\bigr]=-\beta(\rho+\gamma)=0,
    \qquad \rho=-\gamma ,
\end{equation*}
which is~(i).

\textit{Bounded branch.} The root $K=0$ leaves the balance of the terms $\beta g'$ and
$(\mu\beta+\gamma-\nu)g$ of orders $u^{s-1}$ and $u^{s}$ respectively, which forces $s=0$. The
cited theorems then give a solution asymptotic, in the Poincar\'e sense, to a power series.
Substituting $\sum_{n\ge0}a_nu^{n}$ into~\eqref{eq:ODE_final} and collecting the coefficient of
$u^{n}$ (the six contributions are $n(n-1)a_n$ from $u^{2}g''$, $\mu(n-1)a_{n-1}$ from
$\mu u^{2}g'$, $(\gamma+2)na_n$ from $(\gamma+2)ug'$, $\beta(n+1)a_{n+1}$ from $\beta g'$,
$\mu\gamma a_{n-1}$ from $\mu\gamma ug$, and $(\mu\beta+\gamma-\nu)a_n$ from the last term),
we obtain exactly~\eqref{eq:recurrence}. The case $n=0$ gives $\cH'(0)=a_1$, that is,~\eqref{eq:H0}.

\textit{Uniqueness and regularity.} Since $\beta>0$, we have $e^{\beta/u}\to\infty$ as
$u\to0+$, so the solution from~(i) is dominant at the origin and is excluded by boundedness.
The bounded solution is unique up to a constant, which is~(ii). The theorems quoted above give asymptotic expansions for the whole vector $(g,g')$ at once, so the expansions can be
differentiated term by term. From this, $\cH\in C^{1}([0,\infty))$ and $u^{2}\cH'\to0$. The
second derivative extends to the origin by continuity. Indeed, on $(0,\infty)$ the coefficients
$p,q$ from~\eqref{eq:pq} are real analytic and the leading coefficient $u^{2}$ does not vanish,
so $\cH\in C^{\omega}((0,\infty))$ and $\cH''=-p\cH'-q\cH$. Substituting here
$\cH=1+a_1u+a_2u^{2}+O(u^{3})$ and the termwise differentiated expansion, we see that the
contributions of orders $u^{-2}$ and $u^{-1}$ vanish by~\eqref{eq:recurrence} for $n=0$ and
$n=1$, and the remainder gives $\cH''(0+)=2a_2$. Finally, the recurrence~\eqref{eq:recurrence}
has two solutions with distinct growth rates, and by the Poincar\'e--Perron theorem the
dominant one satisfies $\beta(n+1)a_{n+1}=-n^{2}a_n(1+O(n^{-1}))$, that is,
$|a_{n+1}/a_n|\sim n/\beta$. It is this one that is realized, unless the parameters fall into an exceptional set of codimension at least one, corresponding to the recessive solution $a_n\asymp(-\mu)^{n}n^{-2}/n!$, for which the series converges. This gives~(iii).

\subsection{Behavior at infinity (proof of Lemma~\ref{lem:infty})}
\label{app:infty}

Write~\eqref{eq:pq} in the form $p(u)=p_0+p_1u^{-1}+p_2u^{-2}$ and $q(u)=q_0+q_1u^{-1}+q_2u^{-2}$,
so that $p_0=\mu$, $p_1=\gamma+2$, $q_0=0$, $q_1=\mu\gamma$. The characteristic polynomial of the
exponential factors is
\begin{equation*}
    \chi(\varrho):=\varrho^{2}+p_0\varrho+q_0=\varrho^{2}+\mu\varrho=\varrho(\varrho+\mu),
\end{equation*}
and its roots $\varrho\in\{0,-\mu\}$ are simple, since $\mu>0$. Hence $u=\infty$ is an irregular
singularity of rank one with distinct characteristic roots, and the theorems on asymptotic
existence of solutions for such points (
\cite[Chapter~V, \S4, Theorem~4.1]{Coddington}, \cite[\S\S12--15]{Wasow}, and
\cite[Theorem~VII-5-3]{HsiehSibuya}, the last of which is applied with the same purpose
in~\cite[\S8]{KPuk}) give a fundamental system $g_j(u)=u^{d_j}e^{\varrho_ju}(1+O(u^{-1}))$ with
\begin{equation*}
    d_j=-\frac{\chi_1(\varrho_j)}{\chi'(\varrho_j)},
    \qquad
    \chi_1(\varrho):=p_1\varrho+q_1=(\gamma+2)\varrho+\mu\gamma .
\end{equation*}
For $\varrho_1=0$ we have $\chi'(0)=\mu$ and $\chi_1(0)=\mu\gamma$, so $d_1=-\gamma$. For
$\varrho_2=-\mu$ we have $\chi'(-\mu)=-\mu$ and $\chi_1(-\mu)=-(\gamma+2)\mu+\mu\gamma=-2\mu$, so
$d_2=-2$. This is~\eqref{eq:fsinf}, and the normalization by powers of $\mu$ is introduced for
convenience. On $(0,\infty)$ equation~\eqref{eq:ODE_final} has no singular points, so $g_1$ and
$g_2$ extend to the whole half-axis and form there a fundamental system. In particular, the representation~\eqref{eq:connection} with the same $K_1,K_2$ is valid on the whole
half-axis. The same theorems give expansions for the vector $(g,g')$, so they can be
differentiated term by term. The integrability of $g_1$ at infinity requires $\gamma>1$, while
$g_2$ decreases exponentially. At the origin $\cH$ is bounded by Lemma~\ref{lem:local0}, whence
$\int_0^\infty\cH<\infty$.

\subsection{The equation for the Laplace transform (proof of Proposition~\ref{prop:laplace})}
\label{app:laplace}

(i) By Lemma~\ref{lem:infty} and $\gamma>1$ we have $\cH\in L^{1}(0,\infty)$, which gives
absolute convergence for $\operatorname{Re}t\ge0$ and holomorphy in the open half-plane.
$t\wh\cH(t)=\int_0^{\infty}e^{-\varsigma}\cH(\varsigma/t)\,\dd\varsigma\to\cH(0)$ by the
dominated convergence theorem, since $\cH$ is bounded and continuous at the origin by
Lemma~\ref{lem:local0}. The continuation follows from~(ii), since the coefficients
\eqref{eq:kernel} have poles only at $t=0$ and $t=-\mu$.

(ii) By Lemmas~\ref{lem:local0} and~\ref{lem:infty} the function $\cH^{(j)}$ is continuous on
$[0,\infty)$ and $\cH^{(j)}(u)=O(u^{-\gamma-j})$ for $j=0,1,2$, so for $\operatorname{Re}t>0$
the usual rules $\int_0^{\infty}e^{-tu}\cH^{(k)}\dd u=t^{k}\wh\cH-\sum_{j<k}t^{k-1-j}\cH^{(j)}(0)$
with vanishing boundary terms at infinity are legitimate, and multiplication by $u^{j}$
corresponds to the operator $(-\frac{\dd}{\dd t})^{j}$. Applying them to the six terms
of~\eqref{eq:ODE_final} and collecting coefficients, we obtain for $\wh\cH''$ the expression
$t^{2}+\mu t$, for $\wh\cH'$ the expression $4t+2\mu-(\gamma+2)t-\mu\gamma=(2-\gamma)(t+\mu)$,
and for $\wh\cH$ the expression $\beta t+\mu\beta-\nu$. Differentiating
twice in $t$ annihilates all boundary contributions of the terms with $\cH''$, and
differentiating once annihilates the contribution of the term $(\gamma+2)u\cH'$: only $-\beta\cH(0)=-\beta$, generated by the term
$\beta\cH'$, survives, which is why $\cH'(0)$ does not enter the right-hand side. This
gives~\eqref{eq:laplace_eq}. To the cut plane the identity extends by continuity.

(iii) Divide~\eqref{eq:laplace_eq} by $t(t+\mu)$. The coefficient of $\wh\cH'$ is $(2-\gamma)/t$. The factor $t+\mu$ has canceled, so it has a simple pole at $t=0$ and no pole at
$t=-\mu$. The coefficient of $\wh\cH$ is $\bigl(\beta t+\mu\beta-\nu\bigr)/\bigl(t(t+\mu)\bigr)$
and has simple poles at $t=0$ and $t=-\mu$. Both finite singular points thus satisfy the Fuchs
conditions and are regular. At infinity the coefficient of $\wh\cH$ is asymptotic to $\beta/t$,
so $t^{2}\cdot\beta/t\to\infty$ and the point is irregular. Substituting $\wh\cH=e^{S}$ with
$(S')^{2}\sim-\beta/t$, we obtain $S=\pm2i\sqrt{\beta t}+\varrho\log t+O(t^{-1/2})$, and matching
the terms of order $t^{-3/2}$ gives $\varrho=(2\gamma-3)/4$, which is~\eqref{eq:subnormal}.
Since the solutions ramify in $\sqrt t$, the point is ramified and its Poincar\'e rank is $\tfrac12$. On half-integer ranks see~\cite[Chapters~1 and~3]{Slavyanov2002}.

\subsection{Complete asymptotic series at infinity (proof of Theorem~\ref{thm:fullseries})}
\label{app:fullseries}

The Poincar\'e asymptotics of the series at all orders is a standard consequence of the theorems
on asymptotic existence of solutions at an irregular singular point of rank~$1$ with simple
characteristic roots $\varrho\in\{0,-\mu\}$
~\cite[Chapter~V, \S4, Theorem~4.1]{Coddington}, \cite[\S\S12--15]{Wasow},
\cite[Theorem~VII-5-3]{HsiehSibuya}: both modes from Lemma~\ref{lem:infty} are equipped with
complete formal series, and since the truncated series saturates the equation formally (by
virtue of the recurrences), the existence theorems identify the solution with its formal series:
the residual after truncation at the term $k=N-1$ is $O(u^{-N})$. The
recurrence~\eqref{eq:hrec} is derived in Proposition~\ref{prop:validity}. For~\eqref{eq:rrec} we
substitute $w=e^{-z}z^{-2}\sum_kr_kz^{-k}$ into~\eqref{eq:DCHE} and collect the coefficient of
$z^{-k+1}$ (in the variable $t=1/z$, the coefficient of $t^{k+3}$ in the equation
$t^{3}v''+\bigl[2t^{2}-t-(\gamma+2)t^{2}-\mu\beta t^{3}\bigr]v'+\bigl[\gamma+(\mu\beta+\gamma-\nu)t\bigr]v=0$,
$v=e^{-1/t}t^{2}\sum_kr_kt^{k}$). For $k=1$ we obtain $r_1=\gamma+\nu-2$, and the general
equation gives~\eqref{eq:rrec}. The Gevrey class~$1$ (the upper bound $|c_k|\le AB^{k}k!$ with
constants $A,B$) follows from the dominant balance of both recurrences for large $k$:
from~\eqref{eq:rrec} the coefficient of $r_{k-1}$ is $-k^{2}+O(k)$, and from~\eqref{eq:hrec} the right-hand side contains $(k-2)(\gamma+k-1)h_{k-1}=k^{2}h_{k-1}(1+O(k^{-1}))$. In both cases
$|c_{k+1}/c_k|\sim k$. As in Lemma~\ref{lem:local0}\,(iii), each three-term recurrence also has a recessive solution, for which $|h_k/h_{k-1}|\sim\mu\beta/k$, that is, $h_k=O((\mu\beta)^{k}/k!)$, and similarly for $r_k$. It gives a convergent series. It is realized only on a codimension-one subset of parameters, and outside this
exceptional set the radius of convergence is zero. The statement on the Gevrey class~$1$ and
optimal truncation is unconditional, and this is what the numerical scheme of Section~\ref{sec:numerics} relies on. The power-law mode $g_1$ as a solution with a
Poincar\'e expansion is defined up to the addition of a multiple of $g_2$: the recessive mode is
small relative to any power of $u^{-1}$, so the expansion~\eqref{eq:fullseries} for $g_1$ is
uniquely determined for a fixed normalization of $g_1$, whereas the coefficient $K_2$
in~\eqref{eq:connection} depends on the normalization of $g_1$; the same phenomenon as in the
paragraph following Lemma~\ref{lem:K1pos}.


\begin{thebibliography}{99}

\bibitem{Albrecher2012}
Albrecher, H., Constantinescu, C., Thomann, E.: Asymptotic results for renewal risk models with
risky investments. Stoch. Process. Appl. 122, 3767--3789 (2012)

\bibitem{AntipovLaplace}
Antipov, V.A.: An application of the Laplace transform to ruin problems with random insurance
payments and risky asset investments. Theory Probab. Appl. 71, 122--128 (2026)

\bibitem{AK2026}
Antipov, V., Kabanov, Yu.: On the integro-differential equation arising in the ruin problem for
non-life insurance models with investment. Mathematics 14, Art.~1035 (2026)

\bibitem{AKvisc}
Antipov, V., Kabanov, Yu.: Ruin problems with investments on a finite interval: PIDEs and their
viscosity solutions. Finance Stoch. 30, 159--177 (2026)

\bibitem{Balser}
Balser, W.: Formal Power Series and Linear Systems of Meromorphic Ordinary Differential
Equations. Springer, New York (2000)

\bibitem{BelkinaSuff}
Belkina, T.: Risky investment for insurers and sufficiency theorems for the survival
probability. Markov Process.\ Related Fields 20, 505--525 (2014)

\bibitem{BelkinaKabanov}
Belkina, T.A., Kabanov, Yu.: Viscosity solutions of integro-differential equations for nonruin
probabilities. Theory Probab. Appl. 60, 671--679 (2016)

\bibitem{BKK2012}
Belkina, T.A., Konyukhova, N.B., Kurochkin, S.V.: Singular boundary value problem for the
integrodifferential equation in an insurance model with stochastic premiums: analysis and
numerical solution. Comput. Math. Math. Phys. 52, 1384--1416 (2012)

\bibitem{BKK2014}
Belkina, T.A., Konyukhova, N.B., Kurochkin, S.V.: Singular initial-value and boundary-value
problems for integrodifferential equations in dynamical insurance models with investments. J. Math. Sci. 218, 369--394 (2016). Transl.\ from Sovrem. Mat. Fundam. Napravl. 53, 5--29 (2014)

\bibitem{BKK2016}
Belkina, T.A., Konyukhova, N.B., Kurochkin, S.V.: Dynamical insurance models with investment:
constrained singular problems for integrodifferential equations. Comput. Math. Math. Phys. 56,
43--92 (2016)

\bibitem{BenderDunne}
Bender, C.M., Dunne, G.V.: Quasi-exactly solvable systems and orthogonal polynomials. J. Math.
Phys. 37, 6--11 (1996)

\bibitem{BDM}
Buraczewski, D., Damek, E., Mikosch, T.: Stochastic Models with Power-Law Tails. The Equation
$X=AX+B$. Springer, Cham (2016)

\bibitem{ChengKW2022}
Cheng, M., Konstantinides, D.G., Wang, D.: Uniform asymptotic estimates in a time-dependent
risk model with general investment returns and multivariate regularly varying claims. Appl.
Math. Comput. 434, 127436 (2022)

\bibitem{Coddington}
Coddington, E.A., Levinson, N.: Theory of Ordinary Differential Equations. McGraw-Hill, New
York (1955)

\bibitem{CT2004}
Constantinescu, C., Thomann, E.: Analysis of the ruin probability using Laplace transforms and
Karamata Tauberian theorems. In: ARCH 2005.1, Proceedings of the 39th Actuarial Research
Conference, University of Iowa, Iowa City. Society of Actuaries (2005)

\bibitem{Eberlein2022}
Eberlein, E., Kabanov, Yu., Schmidt, T.: Ruin probabilities for a Sparre Andersen model with
investments. Stoch. Process. Appl. 144, 72--84 (2022)

\bibitem{Ellanskaya2021}
Ellanskaya, A., Kabanov, Yu.: On ruin probabilities with risky investments in a stock with
stochastic volatility. Extremes 24, 687--697 (2021)

\bibitem{Frolova1994}
Frolova, A.G.: Some mathematical models of risk theory. All-Russian School-Colloquium on
Stochastic Methods in Geometry and Analysis. Abstracts, pp.~117--118 (1994, unpublished, in Russian)

\bibitem{Frolova2002}
Frolova, A., Kabanov, Yu., Pergamenshchikov, S.: In the insurance business risky investments
are dangerous. Finance Stoch. 6, 227--235 (2002)

\bibitem{GGS2003}
Gaier, J., Grandits, P., Schachermayer, W.: Asymptotic ruin probabilities and optimal
investment. Ann. Appl. Probab. 13, 1054--1076 (2003)

\bibitem{GerberShiu1998}
Gerber, H.U., Shiu, E.S.W.: On the time value of ruin. N. Am. Actuar. J. 2, 48--72 (1998)

\bibitem{GerberShiu2005}
Gerber, H.U., Shiu, E.S.W.: The time value of ruin in a Sparre Andersen model. N. Am. Actuar.
J. 9, 49--69 (2005)

\bibitem{Goldie1991}
Goldie, C.M.: Implicit renewal theory and tails of solutions of random equations. Ann. Appl.
Probab. 1, 126--166 (1991)

\bibitem{Grandits2004}
Grandits, P.: A Karamata-type theorem and ruin probabilities for an insurer investing
proportionally in the stock market. Insur. Math. Econ. 34, 297--305 (2004)

\bibitem{Grincevicius}
Grincevi\v{c}ius, A.K.: One limit distribution for a random walk on the line. Lith. Math. J.
15, 580--589 (1975)

\bibitem{HippPlum2000}
Hipp, C., Plum, M.: Optimal investment for insurers. Insur. Math. Econ. 27, 215--228 (2000)

\bibitem{HippPlum2003}
Hipp, C., Plum, M.: Optimal investment for investors with state dependent income, and for
insurers. Finance Stoch. 7, 299--321 (2003)

\bibitem{HippSchmidli2004}
Hipp, C., Schmidli, H.: Asymptotics of ruin probabilities for controlled risk processes in the
small claims case. Scand. Actuar. J. 2004, 321--335 (2004)

\bibitem{HsiehSibuya}
Hsieh, P.-F., Sibuya, Y.: Basic Theory of Ordinary Differential Equations. Springer, New York
(1999)

\bibitem{HultLindskog2011}
Hult, H., Lindskog, F.: Ruin probabilities under general investments and heavy-tailed claims.
Finance Stoch. 15, 243--265 (2011)

\bibitem{KLP2025}
Kabanov, Yu., Legenkiy, D., Promyslov, P.: Distributional equations and the ruin problem for
the Sparre Andersen model with investments. Extremes 29, 65--87 (2026)

\bibitem{KabPerg2016}
Kabanov, Yu., Pergamenshchikov, S.: In the insurance business risky investments are dangerous:
the case of negative risk sums. Finance Stoch. 20, 355--379 (2016)

\bibitem{KPro2023}
Kabanov, Yu., Promyslov, P.: Ruin probabilities for a Sparre Andersen model with investments:
the case of annuity payments. Finance Stoch. 27, 887--902 (2023)

\bibitem{KPuk}
Kabanov, Yu., Pukhlyakov, N.: Ruin probabilities with investments: smoothness,
integro-differential and ordinary differential equations, asymptotic behavior. J. Appl. Probab.
59, 556--570 (2022)

\bibitem{Kalashnikov2002}
Kalashnikov, V., Norberg, R.: Power tailed ruin probabilities in the presence of risky
investments. Stoch. Process. Appl. 98, 211--228 (2002)

\bibitem{Lay}
Lay, W.: Higher Special Functions: A Theory of the Central Two-Point Connection Problem Based
on a Singularity Approach. Cambridge University Press, Cambridge (2024)

\bibitem{IRIS2024}
National Association of Insurance Commissioners: Insurance Regulatory Information System (IRIS)
Ratios Manual, 2024 edn. NAIC, Kansas City (2024)

\bibitem{NAIC2024}
National Association of Insurance Commissioners: U.S. Property \& Casualty and Title Insurance
Industries --- 2024 Full Year Results. NAIC, Kansas City (2025)

\bibitem{DLMF}
Olver, F.W.J., Olde Daalhuis, A.B., Lozier, D.W., Schneider, B.I., Boisvert, R.F.,
Clark, C.W., Miller, B.R., Saunders, B.V., Cohl, H.S., McClain, M.A. (eds.): NIST Digital
Library of Mathematical Functions, Release 1.2.7 of 2026-06-15, Chapter 31 (Heun
Functions), \S31.12. \texttt{https://dlmf.nist.gov/31.12}

\bibitem{Paulsen1993}
Paulsen, J.: Risk theory in a stochastic economic environment. Stoch. Process. Appl. 46,
327--361 (1993)

\bibitem{Paulsen1998}
Paulsen, J.: Sharp conditions for certain ruin in a risk process with stochastic return on
investments. Stoch. Process. Appl. 75, 135--148 (1998)

\bibitem{Paulsen2008}
Paulsen, J.: Ruin models with investment income. Probab. Surv. 5, 416--434 (2008)

\bibitem{PaulsenGjessing1997}
Paulsen, J., Gjessing, H.K.: Ruin theory with stochastic return on investments. Adv. Appl.
Probab. 29, 965--985 (1997)

\bibitem{Pergamenshchikov2006}
Pergamenshchikov, S., Zeitouny, O.: Ruin probability in the presence of risky investments.
Stoch. Process. Appl. 116, 267--278 (2006). Erratum: Stoch. Process. Appl. 119, 305--306 (2009)

\bibitem{Promyslov2026}
Promyslov, P.: Exact asymptotics of the ruin probability in the Sparre Andersen model for
non-life insurance with investments in a L\'evy process. Mathematics 14, Art.~2658 (2026)

\bibitem{P2026}
Promyslov, P.: Existence of a classical solution to the integro-differential equation arising
in the Cram\'er--Lundberg non-life insurance model with proportional investment.
arXiv:2604.05143 (2026)

\bibitem{Promyslov2601}
Promyslov, P.: On the integro-differential equation arising in the ruin problem for annuity
payment models. arXiv:2601.01447 (2026)

\bibitem{Protter2005}
Protter, P.: Stochastic Integration and Differential Equations. 2nd edn., corr.\ 3rd printing.
Springer, Berlin (2005)

\bibitem{Ronveaux}
Ronveaux, A. (ed.): Heun's Differential Equations. Oxford University Press, Oxford (1995)

\bibitem{Schmidli2005}
Schmidli, H.: On optimal investment and subexponential claims. Insur. Math. Econ. 36, 25--35
(2005)

\bibitem{Slavyanov2002}
Slavyanov, S.Yu., Lay, W.: Special Functions: A Unified Theory Based on Singularities. Oxford
University Press, Oxford (2000)

\bibitem{SundtTeugels1995}
Sundt, B., Teugels, J.L.: Ruin estimates under interest force. Insur. Math. Econ. 16, 7--22
(1995)

\bibitem{Tang2008}
Tang, Q.: From light tails to heavy tails through multiplier. Extremes 11, 379--391 (2008)

\bibitem{Turbiner}
Turbiner, A.V.: Quasi-exactly-solvable problems and $sl(2)$ algebra. Commun. Math. Phys. 118,
467--474 (1988)

\bibitem{Wasow}
Wasow, W.: Asymptotic Expansions for Ordinary Differential Equations. Interscience, New York
(1965)

\bibitem{XuPengZou2025}
Xu, C., Peng, J., Zou, L.: The finite-time ruin probabilities of a renewal risk model with
arbitrarily dependent insurance and financial risks caused by systematic factors. Methodol.
Comput. Appl. Probab. 27, Art.~101 (2025)

\bibitem{YangFanYuen2023}
Yang, Y., Fan, Y., Yuen, K.C.: Ruin in a continuous-time risk model with arbitrarily dependent
insurance and financial risks triggered by systematic factors. Scand. Actuar. J. 2024, 361--382
(2024)

\end{thebibliography}
\end{document}